\documentclass[12pt]{amsart}
\usepackage{amsmath}
\usepackage{amsxtra}
\usepackage{amscd}
\usepackage{amsthm}
\usepackage{amsfonts}
\usepackage{amssymb}
\usepackage{eucal}
\usepackage{color}
\usepackage[all]{xy} 
\usepackage{mathabx}

\newtheorem{thm}{Theorem} [section]
\newtheorem{prop}[thm]{Proposition}
\newtheorem{lem}[thm]{Lemma}

\newtheorem{rmk}[thm]{Remark}
\newtheorem{conj}[thm]{Conjecture}

\newtheorem{Thm}[thm]{Theorem}
\newtheorem{Cor}[thm]{Corollary}
\newtheorem{Lem}[thm]{Lemma}
\newtheorem{Prop}[thm]{Proposition}

\newtheorem{Claim}[thm]{Claim}

\newtheorem{Rem}[thm]{Remark}
\newtheorem{Def}[thm]{Definition}

\newcommand{\nc}{\newcommand}

\numberwithin{equation}{section}

\nc{\fa}{{\mathfrak a}}
\nc{\fb}{{\mathfrak b}}
\nc{\fc}{{\mathfrak c}}
\nc{\fd}{{\mathfrak d}}
\nc{\fe}{{\mathfrak e}}
\nc{\ff}{{\mathfrak f}}
\nc{\fg}{{\mathfrak g}}
\nc{\fh}{{\mathfrak h}}
\nc{\fiI}{{\mathfrak i}}  
\nc{\ffi}{{\mathfrak i}}  
\nc{\fj}{{\mathfrak j}}
\nc{\fk}{{\mathfrak k}}
\nc{\fl}{{\mathfrak{l}}}
\nc{\fm}{{\mathfrak m}}
\nc{\fn}{{\mathfrak n}}
\nc{\fo}{{\mathfrak o}}
\nc{\fp}{{\mathfrak p}}
\nc{\fq}{{\mathfrak q}}
\nc{\fr}{{\mathfrak r}}
\nc{\fs}{{\mathfrak s}}
\nc{\ft}{{\mathfrak t}}
\nc{\fu}{{\mathfrak u}}
\nc{\fv}{{\mathfrak v}}
\nc{\fw}{{\mathfrak w}}
\nc{\fz}{{\mathfrak z}}
\nc{\fx}{{\mathfrak x}}
\nc{\fy}{{\mathfrak y}}

\nc{\fS}{{\mathfrak S}}

\nc{\cA}{{\mathcal A}}
\nc{\cB}{{\mathcal B}}
\nc{\cC}{{\mathcal C}}
\nc{\cD}{{\mathcal D}}
\nc{\cE}{{\mathcal E}}
\nc{\cF}{{\mathcal F}}
\nc{\cG}{{\mathcal G}}
\nc{\cH}{{\mathcal H}}
\nc{\cI}{{\mathcal I}}
\nc{\cJ}{{\mathcal J}}
\nc{\cK}{{\mathcal K}}
\nc{\cL}{{\mathcal L}}
\nc{\cM}{{\mathcal M}}
\nc{\cN}{{\mathcal N}}
\nc{\cO}{{\mathcal O}}
\nc{\cP}{{\mathcal P}}
\nc{\cQ}{{\mathcal Q}}
\nc{\cR}{{\mathcal R}}
\nc{\cS}{{\mathcal S}}
\nc{\cT}{{\mathcal T}}
\nc{\cU}{{\mathcal U}}
\nc{\cV}{{\mathcal V}}
\nc{\cW}{{\mathcal W}}
\nc{\cZ}{{\mathcal Z}}
\nc{\cX}{{\mathcal X}}
\nc{\cY}{{\mathcal Y}}

\nc{\bA}{{\mathbb A}}
\nc{\bB}{{\mathbb B}}
\nc{\bC}{{\mathbb C}}
\nc{\bD}{{\mathbb D}}
\nc{\bE}{{\mathbb E}}
\nc{\bF}{{\mathbb F}}
\nc{\bG}{{\mathbb G}}
\nc{\bH}{{\mathbb H}}
\nc{\bI}{{\mathbb I}}
\nc{\bJ}{{\mathbb J}}
\nc{\bK}{{\mathbb K}}
\nc{\bL}{{\mathbb L}}
\nc{\bM}{{\mathbb M}}
\nc{\bN}{{\mathbb N}}
\nc{\bO}{{\mathbb O}}
\nc{\bP}{{\mathbb P}}
\nc{\bQ}{{\mathbb Q}}
\nc{\bR}{{\mathbb R}}
\nc{\bS}{{\mathbb S}}
\nc{\bT}{{\mathbb T}}
\nc{\bU}{{\mathbb U}}
\nc{\bV}{{\mathbb V}}
\nc{\bW}{{\mathbb W}}
\nc{\bZ}{{\mathbb Z}}
\nc{\bX}{{\mathbb X}}
\nc{\bY}{{\mathbb Y}}

\nc{\on}{\operatorname}

\nc{\BB}{{\mathcal B}}

\newcommand{\oplusl}{\bigoplus\limits}
\newcommand{\cupl}{\bigcup\limits}

\newcommand{\Hom}{{\rm Hom}}

\nc{\od}{{\on{D}}}

\newcommand{\tii}{\widetilde}

\newcommand{\imbed}{\hookrightarrow}
\newcommand{\iso}{{\tii \longrightarrow}}

\newcommand{\To}{\longrightarrow}

\newcommand{\Lotimes}{\overset{\rm L}{\otimes}}

\def\square{\hbox{\vrule\vbox{\hrule\phantom{o}\hrule}\vrule}}

\newcommand{\Nt}{{\widetilde{\mathcal N}}}
\renewcommand{\O}{{\mathcal O}}

\newcommand{\LL}{{\mathcal L}}

\newcommand{\fB}{{\mathfrak B}}

\newcommand{\chiti}{\widetilde{\chi}}

\newcommand{\fBh}{\widehat{\mathfrak B}}

\newcommand{\h}{{\mathfrak h}}

\newcommand{\g}{{\mathfrak g}}
\newcommand{\bu}{\bullet}

\newcommand{\Pone}{{\mathbb P}^1}
\newcommand{\Aone}{{\mathbb A}^1}

\newcommand{\epf}{\square}

\newcommand{\Ah}{\hat{\mathcal A}}

\newcommand{\Tt}{\widetilde{T}}

\newcommand{\A}{{\mathcal A}}
\newcommand{\B}{{\mathcal B}}
\newcommand{\C}{{\mathcal C}}
\newcommand{\F}{{\mathcal F}}
\newcommand{\G}{{\mathcal G}}
\newcommand{\M}{{\mathcal M}}
\renewcommand{\P}{{\mathcal P}}
\newcommand{\LM}{\widecheck{\mathcal M}}
\newcommand{\LD}{\widecheck{\mathcal D}}

\newcommand{\PP}{{\mathcal P}}

\newcommand{\Mh}{\widehat{\mathcal M}} 

\renewcommand{\a}{{\mathfrak a}}

\newcommand{\p}{{\mathfrak p}}
\renewcommand{\k}{{\mathfrak k}}
\renewcommand{\t}{{\mathfrak t}}

 \newcommand{\Lg}{{\check \g}}
 \newcommand{\Lk}{{\check \k}}
 \newcommand{\Lc}{{\check \fc}}

\newcommand{\Zet}{{\mathbb Z}}
\newcommand{\Ce}{{\mathbb C}}
\newcommand{\RE}{{\mathbb R}}

\newcommand{\bs}{\backslash}

 \newcommand{\Mmix}{\M^{\rm mix}}

\newcommand{\Fqbar}{{\bar{\mathbb F}_q}}

\newcommand{\al}{\alpha}

\newcommand{\la}{\lambda}

\newcommand{\LC}{{\check C}}
\newcommand{\LG}{{\check G}}
\newcommand{\LH}{{\check H}}
\newcommand{\LK}{{\check K}}
\newcommand{\LT}{{\check T}}

\newcommand{\LB}{{\check B}}
\newcommand{\LA}{{\check A}}

\newcommand{\LBB}{\check{\mathcal B}}
\newcommand{\LPP}{\check{\mathcal P}}

\newcommand{\Lth}{{\check\theta}}

\newcommand{\LW}{{\check W}}

\newcommand{\Lh}{{\check\h}}
\newcommand{\Lt}{{\check\t}}

\newcommand{\GR}{{G_{\RE}}}

\newcommand{\iR}{R_i}

\newcommand{\Wb}{{W_\M}}
\newcommand{\Wbp}{{W'_\M}}
\nc{\gr}{{\operatorname{gr}}}
\nc{\mhm}{{\operatorname{MHM}}}

\nc{\trans}[2]{{\on T_{#1}^{#2}}}
\nc{\etrans}[2]{{\widetilde{ \on T}_{#1}^{#2}}}

\nc{\flag}{{X}}
\nc{\hla}{{\widehat\la}}
\nc{\eflag}{{\tilde X}}
\nc{\Spec}{{\on{Spec}}}
\nc{\Ad}{{\on{Ad}}}
\nc{\Mod}{{\on{Mod}}}
\nc{\Vect}{{\on{Vect_\bC}}}
\nc{\coker}{{\on{coker}}}
\nc{\hatt}{\widehat}
\nc{\tildemod}{{\widetilde\Mod(\fg,K)}}
\nc{\perv}[2]{{\on P_{#1}(X)_{#2}}}
\nc{\der}[2]{{\on D_{#1}(X)_{#2}}}
\nc{\Coh}{{\cC{\it{oh}}}}
 \nc{\MM}{{\mathcal M}}

\newcommand{\EE}{\cE}
\newcommand{\CC}{{\mathcal{C}}}
\newcommand{\sC}{{\mathcal{S}^{gr}}}
\newcommand{\sO}{{\mathcal{S}}}

\newcommand{\rank}{{\mathrm{rank}\, }}

\newcommand{\codim}{{\mathrm{codim}\, }}

\newcommand{\supp}{{\mathrm{supp}\, }}

\newcommand{\wt}{\widetilde}

\author{Roman Bezrukavnikov}
\address{Department of Mathematics, Massachusetts Institute of Technology, Cambridge, MA 02139, USA}
         \email{bezrukav@math.mit.edu}

\thanks{Roman Bezrukavnikov was supported in part by the NSF grant DMS-1601953,
the US-Israel Binational Science Foundation grant 2016363, the Simons Foundation and by grants from the Institute for Advanced Study and
Carnegie Corporation of New York.}

\author{Kari Vilonen}

\address{School of Mathematics and Statistics, University of Melbourne, VIC 3010, Australia, and Department of Mathematics and Statistics, University of Helsinki, P.O.Box 68, FI-00014 Helsinki, Finland}
         \email{kari.vilonen@unimelb.edu.au and
 kari.vilonen@helsinki.fi}

\thanks{ 
Kari Vilonen was supported in part by the ARC grants DP150103525  and DP180101445, the Academy of Finland,  the Humboldt Foundation, the Simons Foundation and the NSF grant DMS-1402928.}

\title{Koszul duality for quasi-split real groups}

\begin{document}

\maketitle

\tableofcontents

\section{Introduction}
The main result of the paper is a proof of a significant special case of Soergel's conjecture \cite{S}
which provides a categorification for Vogan's character duality \cite{ic4}.\footnote{{\em This 
version contains a post-publication correction, replacing the erroneous Lemma 3.9 from the published paper.
All other statements remain unchanged.}}

Recall that {\em Koszul duality formalism} has been introduced to representation theory
by Beilinson, Ginzburg and Soergel \cite{S0}, \cite{BGS}. As these works demonstrate,
deep numerical information about the Bernstein-Gelfand-Gelfand category $O$ 
of highest weight representations of a complex semisimple Lie algebra $\g$, including Kazhdan-Lusztig conjectures, 
is related to Koszul property of a finite dimensional ring $A$ "controlling" the category. Moreover,
a remarkable result of \cite{S0}, further generalized in \cite{BGS}, asserts that $A$
is in fact isomorphic to its Koszul dual ring $A^!$. The category $O$ is well known
to be equivalent to the category of Harish-Chandra bimodules. This is a special case
of the category of Harish-Chandra $(\g,K)$-modules, where $\g$ and $K$ are, respectively, complexifications of the Lie algebra of a real reductive group and  of its maximal compact subgroup.   A conjectural generalization of the result of \cite{BGS} to the latter setting proposed in
\cite{S} provides a  categorical upgrade of Vogan's character duality \cite{ic4}. In the present article we prove this conjecture in the special case when the real
group in question is quasi-split.

To present further details we need some notation. Let 
$G_\RE$ be  a real semisimple algebraic group and 
 $K_\RE\subset G_\RE$  a maximal compact subgroup. We write 
$G, K$ for the complexifications of
 $G_\RE$, $K_\RE$, respectively, and $\g$, $\k$ for  the Lie algebras of $G$ and $K$.
We fix  a block $\M$ in the category of Harish-Chandra 
$(\g,K)$-modules and we assume that the generalized infinitesimal central character of modules
in $\M$ is regular and integral.
We also fix a strong real form for $G$ with underlying real form $G_\RE$.

Let $\LG$ be Langlands dual complex group. 
Vogan~\cite{ic4} assigned to the above data a (strong) real form of $\LG$
and a block $\LM$ in the category of $(\Lg, \check K)$-modules with a regular integral
infinitesimal central
 character\footnote{The infinitesimal central character is  defined up to translation by the infinitesimal
  central character
of a finite dimensional $\LG$-representation, but the equivalence class of the block is well defined.}.
We impose an additional assumption that $\LM$ is
a {\em principal block}, i.e., that $\LM$ contains a finite dimensional 
$\LG$-representation. This implies that $G_\RE$
is quasi-split. 

By the Beilinson-Bernstein Localization Theorem we can realize $\LM$ as a category
of twisted $\LK$-equivariant $D$-modules on the flag variety $\LBB=\LG/\LB$. Moreover, since $\LM$ 
contains a finite dimensional representation, we can (and will) realize $\LM$ as a subcategory
of $\LK$ equivariant $D$-modules on $\LG/\LB$ (i.e. in fact no twisting is needed).
Let $\LD$ be the full subcategory in the equivariant derived category 
$D_\LK(\LG/\LB)$ consisting of complexes whose cohomology belongs to $\LM$.

The central result of this paper is a {\em   Koszul duality} between
the categories $D^b(\M)$ and $\LD$ conjectured by Soergel~\cite{S}.
 As observed in \cite{S} and recalled below (see Remark \ref{cost_to_st}), this categorifies
Vogan's character duality in~\cite{ic4}  which is
a canonical isomorphism between the Grothendieck group $K(\M)$
and the dual space to $K(\LM)$ implying an equality between Kazhdan-Lusztig 
polynomials for $\MM$ and inverse Kazhdan-Lusztig polynomials for $\LM$. Both Vogan's isomorphism and Soergel's
conjectural equivalence have been stated more generally, when the block $\LM$
is not necessarily principal and $\GR$ may not be quasi-split.
 The present work generalizes results and methods of Soergel's
 \cite{S0} which treated the case of a complex group.

In more detail, we establish a certain relationship between the {\em triangulated} categories
on the two Langlands dual sides: the derived category of the block $\M$ for modules
with {\em generalized} integral regular infinitesimal character for the quasi-split real group $G_\bR$
and the block $\LD$ in the {\em equivariant derived category} $D_{\LK}(\LBB)$.
Namely, we show that the abelian category $\M$ is equivalent to the category
of nilpotent modules over the ring $A=Ext^\bu(\oplus L_i,\oplus L_i)$, where the $L_i$
run over the set of isomorphism classes of irreducible perverse sheaves in the block $\LD$.

By a standard formality argument this implies an equivalence of appropriate graded
versions of the triangulated categories. Observe that $D_{\LK}(\LBB)$ can be thought of
as a {\em derived version} of the category of $(\Lg,\LK)$-modules with a fixed
infinitesimal central character (when tensoring with the base field over the center 
of the enveloping algebra $U(\g)$ one needs to take into account higher Tor's, see
section \ref{sect4} for details). Thus in general our result does not yield two Koszul dual rings
controlling the representation categories on the two sides; however,
we do get such a picture in special cases, such as the previously known case of a complex
group and the new (to our knowledge) case of the principal block in a split group.

Let us recall the mechanism by which the categorical equivalence yields a {\em duality}
at the level of Grothendieck groups (as already explained in \cite{S}). By elementary
 algebra the Grothendieck group of nilpotent $A$-modules $K^0(A-mod_{nilp})$
 is dual to the Grothendieck group of all finitely generated $A$-modules
 $K^0(A-mod_{fg})$, 
 the dual bases in the two groups are provided by the classes of  
 graded irreducible modules and  
 indecomposable graded projective modules,  respectively.
Our result explained above implies that $K^0(A-mod_{nilp})\cong K^0(\M)$, while
the Grothendieck group of the block in $D_{\LK}(\LBB)$ is identified with
$K^0(A-mod_{fg})$
by sending the class of an irreducible $L_j\in D_{\LK}(\LBB)$ to the module
$Ext^\bu (\bigoplus\limits_i L_i, L_j)$; thus it yields a duality between the two 
Grothendieck groups.

We finish the introduction by indicating the key idea of our method.
Recall that a central role in Soergel's theory is played by the {\em Soergel bimodules},
which form a full subcategory in the category of coherent sheaves on $\fh^*\times _{\fh^*/W}\fh^*$,
where $\fh$ is the abstract Cartan algebra of $\g$.
An analogous role in our construction is played by the so called
{\em block variety}, $\fB=\a^*/\Wbp\times _{\fh^*/W} \fh^*$, where $\a$ is a maximal split torus in $\g$ and
 $\Wbp\subset W$ is a subgroup depending on the block $\M$. 
 The appearance of the quotient $\a^*/\Wbp$ can be motivated as follows.
 On the one hand it (or rather its completion at zero) can be realized as the space parametrizing
 deformations of an irreducible principal series module at the singular infinitesimal central character $-\rho$,
 which belongs to the image of $\M$ under the translation functor.
 On the other hand,
  the ring $\O(\fB)$ is identified with the equivariant cohomology ring 
  $H^\bu_{\LK}(\LBB)$. The argument proceeds then by describing the 
  categories of projective pro-objects in $\M$, as well as the subcategory
  of semisimple complexes in the block of $D_\LK(\LBB)$ as full subcategories
  in the coherent sheaves on the block variety (or rather its completion at zero).
  In the case when the group is not adjoint one needs to modify the above strategy by
considering coherent sheaves on the block variety equivariant with respect to a certain
finite abelian group.
  
In this paper we use a generalization of the by now classical approach initiated in \cite{S0}, \cite{BGS}, combining it with combinatorial information from Vogan's work \cite{ic4}.
Ben-Zvi  and Nadler \cite{BZN} have proposed a way to approach Soergel's conjectures via geometric Langlands duality. It would be interesting to  relate the two approaches. 
We expect that the relation of our construction to Hodge $D$-modules (see
section  \ref{sect5}) may provide a starting point for this line of investigation.

\bigskip

The paper is organized as follows. 

In section \ref{sect1} we recall some combinatorial invariants of our categories.
Most of these appear in \cite{green_book}, but we also use \cite{AdC}
as a convenient reference. 

In section \ref{sect3} we give a description of $\LD$ based on the functor
of  cohomology (derived global sections) which turns out to be faithful
on the subcategory of irreducible objects.

In section \ref{sect2} we provide a description of category $\M$
based on the functor of translation to singular central character which turns
out to be faithful on the subcategory of projective (pro)objects.

In section \ref{sect4} we formally state and prove our Koszul duality result by comparing
the two descriptions. We end the paper by section \ref{sect5}  where 
we present a  conjecture on realization of $\M$ in terms of coherent sheaves on 
the cotangent bundle via  Hodge
$D$-modules and explain its relation to our description of $\M$. 

{\bf Acknowledgements.} We are much indebted to 
 David Vogan for many helpful explanations over the years; in particular,
 the key "matching" Theorem \ref{thm_match} was explained to us by him.
We also thank Jeff Adams and Peter Trapa for useful discussions.

\section{Combinatorics of a block}\label{sect1}

In this section we recall the basic facts about Vogan duality \cite{ic4,green_book,AdC} and its underlying combinatorics in the setting explained in the introduction. The duality involves making certain choices and part of the statement is that those choices can be made consistently. 

\subsection{The equivariant (principal block) side}
\label{equivariant}
The combinatorial data related to the block $\LM$ is as follows. Let $\LC$ be a maximal torus
in $\LK$ and $\LT=Z_\LG(\LC)$ be the centralizer of $\LC$ in  $\LG$: it is a standard fact
that $\LC$ contains elements regular in $\LG$ (this is equivalent to absence of real roots
in this case, see e.g. \cite[Proposition 6.70]{Knap}),
thus $\LT$ is a maximal torus in $\LG$. Let $W(\LK)=W(\LK,\LC)=N_{\LK}(\LC)/\LC$
be the Weyl group. 

We have an involution $\Lth$ on $\LG$ with $\LG^{\Lth}=\LK$. As the torus $\LT$ is 
stable under $\Lth$, 
it induces an involution of the Weyl group $W(\LG,\LT)=N(\LT)/\LT$. 
We write $W(\LG,\LT)^\Lth$ for the subgroup  of $W(\LG,\LT)$ fixed by this involution.

We can choose Borel subgroups $B_\LK\subset B_\LG$ of $\LK$ and $\LG$ containing $\LC$,
then $B_\LG$ contains $\LT$ so this choice yields an identification between
$\LT$ and the abstract Cartan group $\LH$. Note that $B_\LG$ is $\Lth$-stable, since $\Lth$ fixes a regular
element in $B_\LG$. It follows  that varying the choice above we get identifications $\LT\cong \LH$ which are permuted 
by the group $W(\LG,\LT)^\Lth$ , in particular, the induced involution on $\LH$ and on the abstract
Weyl group $\LW$ of $\LG$ is independent of the choice.
This involution will also be denoted by $\Lth$. 

Note that the group $\LK$ may be disconnected. We have a homomorphism
$W(\LK)\to \pi_0(\LK)$ which is onto. Its kernel is the group $W(\LK^0)=W(\LK^0,\LC)$ where $\LK^0$
is the identity connected component in $\LK$.  Note that the group $\pi_0(\LK)\cong\pi_0(\LC)$ is an elementary abelian 2-group, i.e., it is isomorphic to $(\Zet/2\Zet)^r$ for some $r$. Thus we have 
$$ 
W(\LK^0)\subset W(\LK)\subset W(\LG,\LT)^\Lth.
$$

Finally, we need a count of the closed $\LK$-orbits on the flag manifold $\LBB=\LG/\LB$. 
Let $(\LK\bs \LBB)_{cl}$ be the set of closed $\LK$ orbits on $\LBB$. 

The set of $\LT$ fixed points carries an action of $W(\LG,\LT)\times \LW$, where the $W(\LG,\LT)$ action comes from the
action of the normalizer of $\LT$ by conjugation, while the action of the abstract Weyl group $\LW$ is such that
the action of a simple reflection $s_\al$ sends a fixed point $x$ to the unique fixed point $y\ne x$ which
has the same image in the partial flag variety $\LPP_\al$ parametrizing minimal parabolics of type $\al$. 
We write the $W(\LG,\LT)$ action as a left action and the $\LW$ action
as a right one. Both actions are simply transitive. Choosing a point $x\in \LBB^\LT$ we get an identification
$W(\LG,\LT)\cong \LW\cong \LBB^\LT$ such that the above action of $W(\LG,\LT)\times \LW$ on $\LBB^\LT$ is given
by $(w,v):x\mapsto wxv^{-1}$.

For an orbit $O\in \LK\bs \LBB$ the intersection $O\cap \LBB^\LT$ is either empty or 
it is  a coset of $W(\LK)$; moreover, the intersection is nonempty if the orbit is closed.
Thus we get an injective map $\phi:(\LK\bs \LBB)_{cl}\to W(\LK)\bs \LBB^\LT$. Note that the $\LW$ action on $\LBB^\LT$ descends to an action
on $W(\LK)\bs \LBB^\LT$. 

\begin{Claim}\label{21}
The map $\phi$ identifies $(\LK\bs \LBB)_{cl}$ with an orbit of the group $\LW^\Lth\subset \LW$.

\end{Claim}

Abusing notation we will write $(\LK\bs \LBB)_{cl}\cong W(\LK)\bs \LW^\theta$ where it is 
understood that the embedding $W(\LK)\to \LW^\Lth$ is well defined only up to conjugacy by 
$\LW^\Lth$.

Let us now consider the above situation on the level of Lie algebras. Then $\Lc$ is a maximal torus in $\Lk$ and $Z_{\check \fg}(\check \fc)=\check \ft$ is a maximal $\Lth$-stable torus in $\check \fg$. We write $\Lh$ for the abstract Cartan algebra of $\Lg$. 
 Any fixed point of the $\Lth$-stable torus $\check T$ on $\LBB$ gives rise to an  identification between $\Lt$ and $\Lh$ and these identifications are related by the Weyl group action. Hence, there is a canonical identification $\Lt/W = \Lh/W$. 
 Thus we obtain a canonical map $\Lc/W(\LK^0) \to \Lh/W$. In this setting we define the block variety as  $\fB_{eq}= \Lc/W(\LK^0)\times _{\Lh/W} \Lh$. The group $\fS_{eq}=W(\LK)/W(\LK^0)$ acts on $\fB_{eq}$ via its natural action on the first factor.

\subsection{The monodromic (quasisplit) side}
\label{quasisplit}
Recall from the introduction that we are considering a quasi-split real
form $G_\RE$ and a fixed block in the category of (finite length) $(\g,K)$-modules
which we denoted by $\M$. All modules in the block have the same
generalized infinitesimal central character which is assumed to be regular
and integral; thus it corresponds to a uniquely defined dominant weight denoted
by $\la$. 

We identify $\M$ with its image under the localization equivalence
between the category of $(\g,K)$-modules and the category of
$K$-equivariant twisted $D_{\widehat{\la}}$-modules on the flag variety $\BB=G/B$ as in \cite{Be1};
here $D_{\widehat{\la}}$ denotes an infinitesimally deformed twisted differential 
operators ring corresponding to the line bundle $\O(\lambda)$.

A ($K$-equivariant) $D_{\widehat{\la}}$ module is  by definition the same as a ($K$-equivariant)
$D$-module on the extended flag manifold $X=G/U$, where $U$ denotes the unipotent radical of $B$,  which is weakly equivariant 
 with respect to the right 
action of the universal Cartan $H\cong B/U$ with the  generalized eigenvalues of the  log monodromy determined
by $\la$. We further identify the latter category, via the Riemann-Hilbert correspondence, with the category of $K$-equivariant $H$-monodromic perverse sheaves on $X$ where the  monodromy is determined by\footnote{Here by monodromy we mean monodromy on the image under the De Rham functor,
rather than on the dual functor of solutions.} 
$\exp(2\pi i \la)$.
 In what follows we mostly work in terms of the (twisted) perverse sheaves.

An irreducible object in $\M$ 
 is of the form $j_{!*}(\LL)$ where $j:O\to X$
is the embedding of a $K\times H$-orbit and $\LL$ is an irreducible
$K$-equivariant and $H$-monodromic local system.

We say that the pair $(O,\LL)$ belongs to the block if $j_{!*}(\LL)\in \MM$.
We let  $\LL_\MM(O)$ denote the set of local systems on $O$ belonging to the block
$\MM$. Let $Irr(\MM)$ be the set of irreducible objects in $\MM$,
thus we have $Irr(\MM)\cong \cupl _{O\in K\bs G/B} \LL_\MM(O)$.

For the rest of this section the orbit $O$ will be the open $K\times H$-orbit $X_0$ on $X$ which is the inverse image of the open $K$-orbit $\BB_0\subset \BB$. 

We have a Cartan involution $\theta$ on $G$ such that $K= G^\theta$. We fix a $\theta$-stable maximal torus $T$ of $G$ which is maximally split, i.e. such that $T^\theta$ is of minimal dimension. In that case $T$ has a fixed point in the open orbit $\BB_0$. Fixing such a point we get an identification between $T$ and the abstract Cartan $H$.
This defines an involution $\theta$ on $H$ which is independent of any choices.
 We let 
$\fa$ be the $(-1)$ eigenspace of $\theta$ acting in $\fh$. 



As $T$  is a maximally split torus the real Weyl group $W_\bR= W(K,T)$ coincides with $W^\theta$. This follows, for example, from~\cite[Propositions 3.12 and 4.16]{ic4} as for $T$ there are no non-compact imaginary roots.  A $T$ fixed point $x\in \BB^T$ determines its $K$-orbit on $\BB$; we say that $K$-orbits arising this way are {\it attached} to $T$. These include the open orbit $\BB_0$. We will use this  terminology in general, i.e., given any $\theta$-stable maximal torus $T'$ we say that $K$-orbits containing  fixed points of $T'$  are {\it attached} to the maximal torus $T'$. 

We recall, as is rather easy to see, that all blocks have representatives on the open orbit, i.e., that $\LL_\cM(X_0)$ is not empty. As the action of $K\times H$ is transitive on $X_0$
and the reductive part of the stabilizer is isomorphic to $T^\theta$, the set $\LL_\MM(X_0)$ is identified with a subset
in the set of characters of $T^\theta$ whose differential is determined by the infinitesimal central character $\la$. The latter set is a torsor over
the set of characters of $\pi_0(Stab_{K\times H}(\tilde{x}))=
\pi_0(Stab_K(x))=\pi_0(T^\theta)$ for a point $x\in \BB_0$ and its lift $\tilde x\in X_0$.
The finite group $\pi_0( T^\theta)$ is isomorphic to a product of a finite number of copies of $\Zet/2\Zet$.

\medskip

We will make use of  the {\em cross-action} of $W$ on the set $Irr(\MM)$,
\cite[Definition 8.3.1]{green}, \cite[sections 9 and 14]{AdC}. We will summarize some properties of the cross action that will be important to us  in section \ref{cross}.

In particular, $W$ acts on the set $\cup \LL_\cM(O)$ where $O$ runs through the orbits attached to the Cartan $T$. The subgroup $W_\bR=W^\theta$ of $W$ 
 preserves the subset $\LL_\cM(X_0)$. We will need the following consequence of \cite[Theorem 8.5]{ic4}. 

\begin{Claim}
\label{Claim 2}
The action of $W^\theta$ on the set $\LL_\cM(X_0)$ is transitive.
\end{Claim}

 The Claim shows that choosing  $\cL_0\in \LL_\cM(X_0)$ and setting $\Wb=Stab_{W^\theta}(\cL_0)$ 
we get a bijection between the set $\LL_\cM(X_0)$ and $W^\theta/\Wb$.

We will also need a comparison between the relevant data for $G$
and for the adjoint group $G_{ad}$ equipped with the compatible real
form. 

Let $K'\subset G$ be the subgroup generated by $K$ and the center $Z(G)$ of $G$.
Note that the finite group $Z(G)^\theta=Z(G)\cap K$ acts on all modules in $\M$ by the same
character which we denote by $\chi$. Fix an extension $\tilde \chi$ of $\chi$
to a character of $Z(G)$. It is easy to see that the category of $(\g,K)$ modules
on which $Z(G)\cap K$ acts by $\chi$ is equivalent to the category of $(\g, K')$
modules where $Z(G)$ acts by $\tilde \chi$. We identify $\M$ with its image under
that equivalence; the localization theorem identifies this category with 
a subcategory in  $K'$-equivariant $D_{\widehat{\la}}$ modules.
In particular, the elements of $\LL_\MM(X_0)$ can be viewed as $K'$-equivariant local
systems on the corresponding orbits. 

Let $K_{ad}$ denote the fixed points of the involution of $G_{ad}$ compatible
with the  one fixed on $G$. Let $\widetilde K_{ad}$ be the preimage of $K_{ad}$ 
under the homomorphism $G\to G_{ad}$. 
 The group $\tilde K_{ad}/K'= K_{ad}/Im(K)$ is abelian, 
we let $\fS_{mon}=(K_{ad}/Im (K))^*$ be the dual
abelian group.

Let $\LL_\MM^{ad}(X_0)$ be the set of $\widetilde K^{ad}$ equivariant local systems
 on $X_0$ which,  viewed as a $K'$-equivariant local system, belong to 
 $ \LL_\MM(X_0)$. 
For $x\in \BB_0$ we have a short exact sequence of abelian groups
$$\{1\}\to \pi_0(Stab_{K'}(x)) \to \pi_0(Stab_{\widetilde K^{ad}}(x))\to \fS_{mon}^* \to \{1\}.$$ 
Thus we have a free action of $\fS_{mon}$ on $\LL_\MM^{ad}(X_0)$, such that
$\LL_\MM^{ad}(X_0)/\fS_{mon} = \LL_\MM(X_0)$.

The cross action lifts to an action of $W^{\theta}$ on $\LL_\MM^{ad}(X_0)$,
which is also transitive.
We write $\Wbp$ for the stabilizer of a point $\cL_0^{ad}$ in $\LL_\MM^{ad}(X_0)$
under this action. We choose the base point $\cL_0^{ad}$ compatibly with the choice of $\cL_0$ fixed in the paragraph
after Claim \ref{Claim 2}.  
With this choice we have  $\Wbp \subset \Wb$.
The considerations above show that we have a short exact sequence of groups:
\begin{equation}\label{Wbprime}
\{1\}\to \Wbp \to \Wb \to \fS_{mon} \to \{1\}.
\end{equation}
Finally, we can define the block variety in terms of these data as $\fB_{mon}=\a^*/\Wbp\times _{\fh^*/W} \fh^*$. The group $\fS_{mon}$ acts on $\fB_{mon}$ via its action on the first factor which comes from the exact sequence  \eqref{Wbprime}. 

\begin{Rem}\label{Bmoncan}
The definition of $\fB_{mon}$ appears to depend on the choice of $\LL_0^{ad}$
which determines the subgroup $\Wbp=\Wbp(\LL_0^{ad})$.  
 However, we have 
 \begin{equation}\label{faWp}
 \fa^* /\Wbp = (\LL_\MM^{ad}(X_0)\times \fa^*)/W^\theta.
 \end{equation}
Here the right hand side is independent of any choices, thus
 we get a canonical description of $\fB_{mon}$.
 
One can think of  the formal completion of $(\LL_\MM^{ad}(X_0)\times \fa^*)/W^\theta$ at 
$(\LL_\MM^{ad}(X_0)\times \{0\})/W^\theta$
as parametrizing formal deformations of a local system in 
$\LL_\cM^{ad}(X_0)$ modulo the action of $W^\theta$ induced by the cross action.
  
It is also easy to see that the $\fS_{mon}$ action on $\fB_{mon}$ is canonical.

Note, however, that the map $\fa^*\to \fB_{mon}$ does depend on the additional choice.
\end{Rem} 

\subsection{Matching}
\label{matching}

In the  previous subsections we have defined two versions of the block variety $\fB_{eq}$ and $\fB_{mon}$ with actions of  $\fS_{eq}$ and $\fS_{mon}$, respectively. We now make use of the combinatorics of Vogan duality~\cite{ic4} to show that these two constructions yield the same variety. This is a crucial ingredient in our arguments. 

In the discussion above we have defined a set of local systems $\LL_\MM^{ad}(X_0)$. We can also think of these local systems as belonging to a block $\cM^{ad}$ of  $(\fg,K_{ad})$-modules, but to do so we have to move to a different infinitesimal character. To this end we choose an irreducible representation $V_\mu$ of highest weight $\mu$ such that the center $Z(G)$ acts on $V_\mu$ by the character $\widetilde \chi^{-1}$. We now make use of the translation functor $T_{\la\to\la+\mu}$  (which is an equivalence of categories, of course) to translate our set up from the generalized infinitesimal character $\la$ to the generalized infinitesimal character $\la+\mu$. After the translation the modules $\LL_\MM^{ad}(X_0)$ belong to a particular block $\cM^{ad}$ of  $(\fg,K_{ad})$-modules. More precisely, we have
\begin{equation*}
T_{\la\to\la+\mu}\ \text{carries the set} \ \LL_\MM^{ad}(X_0) \ \text{to the set}\  \LL_{\MM^{ad}}(X_0)\,.
\end{equation*}

After the translation, the blocks $\cM$ and $\cM^{ad}$ can be compared directly as the modules in $\cM$ consist now of certain $(\fg,Im(K))$-modules. 

Recall that fixing a $K$ orbit on $G/B$ defines an involution of the abstract Cartan $\h$ and similarly for $\LG$;
the involution $\theta$ of $\h$ comes from the open $K$ orbit, while the involution $\Lth$ comes from
any one of the closed orbits of $\LK$, it is well known that $\Lth$ does not depend on the choice of closed
orbit.
We have:

\begin{Thm}\label{thm_match}
a) The  canonical isomorphism of abstract Cartan Lie algebras 
$\h^*\cong \Lh$ sends the  involution $\theta$ to $-\Lth$.

b) The isomorphism $W^\theta \cong \LW^\Lth$
arising from (a) sends the conjugacy class of the
pair of subgroups $\Wbp\subset \Wb$ to that of 
the pair $W(\LK^0)\subset W(\LK)$.  

\end{Thm}

\proof
This is a direct consequence of the combinatorial matching underlying Vogan character duality \cite{ic4}, see also \cite{AdC}.
As follows from \cite[Corollary 10.8]{AdC}  we have a canonical bijection between the sets of irreducible objects  $Irr(\MM)$ and $Irr(\LM)$. 
It is compatible with the cross action of $W$. Notice that the stabilizer of point in a closed $\LK$ orbit on $\LBB$
is connected, so an equivariant local system on such an orbit is trivial, thus we have an embedding
$(\LK\bs \LBB)_{cl}\subset Irr(\LM)$; the action of $\LW^\Lth$ on $(\LK\bs \LBB)_{cl}$ in Claim \ref{21}
is easily seen to be the restriction of the cross action.

 The bijection between $Irr(\MM)$ and $Irr(\LM)$ satisfies the following additional compatibility.
An  object $L\in Irr(\MM)$ defines a $K$-orbit on $\BB$ which in turn defines an involution $\theta_L$
on the abstract Cartan $\h$, and similar for $\LM$. For matching elements $L\in Irr(\MM)$ and $\check{L}\in Irr(\LM)$
we have: $\theta_L=-\theta_{\check{L}}$ where we have identified automorphisms of $\Lh$ and of $\Lh=\h^*$.
It is easy to see that for $G$ quasisplit a $K$ orbit in $\BB$ is open if and only if the corresponding involution
sends positive roots to negative roots. Similarly, an orbit is closed if and only if the corresponding involution
sends positive roots to positive roots. Thus the bijection identifies $\cL_\MM(X_0)$  with $(\LK\bs \LBB)_{cl}$.
This implies (a) and the part of (b) involving $\Wb$ and $W(\LK)$. The matching between $\Wbp$
and $W(\LK^0)$ follows from existence of a compatible bijection for the corresponding blocks for
the simply connected cover  of $\LG$ and the adjoint quotient $G_{ad}$. \qed

\begin{Rem}
The isomorphism of Theorem \ref{thm_match} induces a canonical isomorphism
between $\fa^*=(\h^*)^{-\theta}$ and $\Lh^\Lth$. Notice that  $\Lh^\Lth\cong \Lc$
noncanonically. 
The resulting isomorphism $\fa^*\cong \Lc$ is defined uniquely up to  action
of $W^\theta$.
\end{Rem}

\begin{Cor} 
We have isomorphisms $\fB_{eq}\cong \fB_{mon}$,
$\fS_{eq}\cong \fS_{mon}$ compatible with the action
of the latter on the former.  
\end{Cor}

\proof
We choose compatible base points $S_0\in (\LK \bs \LB)_{cl}$,  $S_0^{sc}\in (\LK^0 \bs \LB)_{cl}$ and also
$\cL_0\in \LL_\cM(X_0)$, $\cL_0^{ad}\in \LL_\MM^{ad}(X_0)$ so that the subgroups  $\Wbp\subset \Wb \subset W^\theta$
correspond to   $W(\LK^0)\subset W(\LK)\subset \LW^\Lth$ under the canonical identification $W^\theta = \LW^\Lth$.

This yields a particular choice of an isomorphism $\fa^*\cong \Lc$ intertwining the action of $\Wb\cong W(\LK)$,
thus inducing an isomorphism 
between $(\fB_{eq},\fS_{eq})$ and $(\fB_{mon},\fS_{mon})$.
 \qed

We use the isomorphism to identify  $(\fB_{eq},\fS_{eq})$ with $(\fB_{mon},\fS_{mon})$ and will denote the resulting pair by $(\fB,\fS)$ from now on.

\medskip

It follows from the next Lemma that
if $G$ is simply connected then the subgroup $\Wb \subset W^\theta$ is the normalizer\footnote{In many but not all
cases it is also true that $\Wb$ is self-normalizing. A counterexample is provided by the principal block for $G_\RE=Sp(4, \RE)$: in this case $W^\theta=W$, while $\Wb$ is generated
by reflections at short roots.}
 of $\Wbp$.
Thus the choice of a compatible matching is unique up to  action of the abelian group $\fS=\Wb/\Wbp$ and the 
isomorphism between block varieties is also unique up to $\fS$ action, resulting in a {\em canonical}
equivalence between the categories of $\fS$-equivariant coherent sheaves.

 \begin{Lem} If $\LG$ has trivial center then $W(\LK,\LC)$ coincides with the normalizer of   
 $W(\LK^0,\LC)$ in  $W(\LG,\LT)^\Lth$.
 \end{Lem}
 
 \proof We will abbreviate $W(\LK,\LC)$,
 $W(\LK^0,\LC)$,  $W(\LG,\LT)$ to  $W(\LK)$,
 $W(\LK^0)$,  $W(\LG)$. 
 It is clear that $W(\LK)$ normalizes $W(\LK^0)$, we only need to check that the normalizer $N=N_{W(\LG)^\Lth}(W(\LK^0))$ is contained in $W(\LK)$.

 By \cite[Proposition 3.12(c)]{ic4} we have $W^\Lth=(W^C)^\Lth \ltimes W^{\iR }$ 
where we use notation of {\em loc. cit.}, in particular $W^{\iR }$ is generated by $s_\al$, $\alpha\in \iR $. The group $(W^C)^\Lth$ is generated by elements $s_\al s_{\Lth(\al)}$
for certain roots  $\al$  such that $\al\ne \Lth(\al)$ and $\al\pm \Lth(\al)$ is not a root. If $\tilde s_\al\in SL(2)_\al$ is a representative for $s_\alpha$ (where $SL(2)_\al$
is the image of the corresponding homomorphism $SL(2)\to \LG$) then 
$\tilde s_\al \Lth(\tilde s_\al)=\Lth(\tilde s_\al)s_\al\in \LK$, so $(W^C)^\Lth \subset W(\LK)$. It remains to show that 
\begin{equation}\label{NWir}
N\cap W^{\iR }\subseteq W(\LK).
\end{equation}
 
 Let $R$ be the set of roots for $\LT$ in $\LG$ and $\iR  =\{\al \ |\ \Lth(\al)=\al\}$ be the subset of imaginary roots.
 The involution $\Lth$ induces a $\Zet/2\Zet$ grading $\epsilon$ on $\iR $ where $\epsilon(\al)=0$ iff $\Lth | _{\Lg_\alpha}=Id$. 
It is well known that under the above assumptions $W(\LK)=Stab_{W(\LG)^\Lth}(\epsilon)$. To prove it one observes that
 for $w\in W(\LG)^\Lth$ with a representative $\tilde w$ in the normalizer of $\LT$ we have
$\tilde w\in W(\LK)$ iff the element $t_{\tilde w}=\theta(\tilde w) \tilde w^{-1}\in \LT$ can be written as $t_{\tilde w}= \theta(t)t^{-1}$ for some $t\in \LT$. If $\tilde w$ preserves the grading then $t_{\tilde w}$ is annihilated
by all imaginary roots. By the choice of $\LT$ there exists a $\Lth$ invariant Borel $\LB\supset \LT$, so there are no real roots (i.e. roots with $\Lth(\al)=-\al$). 
We can choose a subset of (complex) simple roots $S$ such that $R=S\sqcup \Lth(S)\sqcup \iR $ and define $t$ by $\alpha(t)=\alpha(t_{\tilde w})^{-1}$ for $\alpha\in S$ and  $\alpha(t)=1$
otherwise, then clearly $t_{\tilde w}$ and $\theta(t)t^{-1}$ have the same pairing with every root, so they are equal since $\LG$ is adjoint. 

We proceed to check \eqref{NWir}. 
It is clear that if $\epsilon(\alpha)=0$ then $s_\al\in W(\LK^0)$. We claim that the converse is true unless
  $\alpha$ is isolated\footnote{It may in fact fail if $\alpha$ is isolated, for  example if  $\LG_\RE=PGL(3,\RE)$ then $\Lth$
  swaps the two simple roots of the $A_2$ root system, while the highest root $\alpha$ satisfies $\theta(\alpha)=\alpha$, $\epsilon(\alpha)=1$, $s_\alpha\in W(\LK^0)$.}
    in $\iR $, i.e.   $\alpha$
is orthogonal to all $\beta\in \iR $,  $\beta \ne \pm \alpha$:
\begin{equation}\label{even_im}
s_\al \in W(\LK^0) \Rightarrow \epsilon(\al)=0 \vee (\al,\beta)=0 \ \forall \beta\in \iR, \beta\ne \pm \al. 
\end{equation}
 If a root $\alpha\in \iR$ is isolated then $\{\pm \alpha\}$ is invariant under $W^{\iR }$, so \eqref{even_im}
 implies \eqref{NWir}. 
To check \eqref{even_im}
observe that if there exists a root $\beta\in \iR $
such that $2\frac{(\alpha,\beta)}{(\alpha,\alpha)}$ is odd then 
$\epsilon(s_\al(\beta))=\epsilon(\al)+\epsilon(\beta)$, so 
$s_\alpha\in Stab(\epsilon)\Rightarrow \epsilon(\al)=0$.
Suppose now that no such $\beta$ exists and $\al$ is not isolated. Then 
 the irreducible summand in $R$ containing $\alpha$ has to be of type $B_n$ ($n\geq 2$) where $\alpha\in B_n$ is short.
We can assume without loss of generality that $\LG$ is simple. Since its Dynkin graph is not simply laced it has no
nontrivial automorphisms, thus $Aut(\LG)=\LG$, in particular $\Lth$ is inner. It follows
that $\LC=\LT$ and the roots of $\LC=\LT$ in $\LK$ are exactly the roots $\alpha\in \iR =R$ with $\epsilon(\al)=0$. 
Thus by the standard theory of Weyl groups of reductive groups, reflections in $W(\LK^0)$ are exactly $s_\al$, $\epsilon(\al)=0$. 
\qed
 
\section{Principal block in the equivariant derived category}
\label{sect3}

In this section we consider a real form of $\LG$ and we write $\LK$ for the complexification of the maximal compact subgroup of the real form. We work with the principal block $\LD$ in the category $\od^b_\LK(\LBB)$ and make use of the other notation introduced in section~\ref{sect1}. In particular, recall that $\Lc$ stands for a maximal torus in $\Lk$, that $Z_{\check \fg}(\check \fc)=\check \ft$ is a maximal $\Lth$-stable torus in $\check \fg$, and that we write $\Lh$ for the abstract Cartan algebra of $\Lg$. Recall also that $W=\check W$, and so we will often denote the Weyl group just by $W$.

We begin by considering the case when the group $\LK$ is connected.  In particular,  ${H^{*}}(\text{pt}/\check K) \cong \Ce[\check \fc]^{W(\check K,\check \fc)}$ and we have a canonical map $\Lc/W(\LK) \to \Lh/W$.

\begin{Lem}
\label{Global equivariant cohomology}
We have a canonical isomorphism:
$H^*_{\LK}(\LBB)=\Ce[\fB].$
\end{Lem}

\begin{proof}
Let us choose a particular Borel $\check B$ in $\check G$. This gives us 
 the following diagram on the level of spaces
\begin{equation}\label{41}
\begin{CD}
\check K\backslash \LG/\check B @>>> \text{pt}/\check K
\\
@VVV @VVV
\\
 \text{pt}/\check B @>>>  \text{pt}/\check G\,.
\end{CD}
\end{equation}
Note that up to homotopy we can replace $ \text{pt}/\check B$ by  $\text{pt}/\check H$ so the diagram itself is independent of the choice of $\LB$.

Passing to cohomology and using the canonical isomorphisms ${H^{*}}(\text{pt}/\check K) \cong \Ce[\check \fc]^{W(\check K)}$,  ${H^{*}}(\text{pt}/\check G) \cong \Ce[\check \fh]^W$, and  ${H^{*}}(\text{pt}/\check H ) \cong \Ce[\check \fh]$ give us a canonical map
\begin{equation*}
\ \Ce[\check \fh]\otimes_ {\Ce[\check \fh]^W}\Ce[\check \fc]^{W(\LK)} \to H^*_{\LK}(\LBB).
\end{equation*}
On the other hand, the Serre spectral sequence associated with the upper horizontal
arrow in \eqref{41} degenerates, as it has zero terms in odd degrees.
This gives the conclusion. 
\end{proof}

\subsection{Global cohomology is fully faithful on simples (statement)}
\label{sect_full}
Since global equivariant cohomology of a space acts on the cohomology (derived global
sections) of any equivariant complex, we get a functor $R\Gamma_\LK: \LD\to \Coh(\fB)$.
Let $\sC\subset \LD$ be the full subcategory of semisimple complexes, i.e.
sums of shifts of simple perverse sheaves. Let $\sO$ be the category whose objects are 
semisimple perverse sheaves in $\LD$ and 
  morphisms are given by $Hom_\sO(L_1,L_2) = Ext^\bu_{\LD}(L_1,L_2)$. 

One of the main results of this section is the following
\begin{Thm}\label{RGamma_full}
Assume that $\LK$ is connected. The functor $R\Gamma_\LK$ induces full embeddings
 $\sO \to \Coh(\fB)$, $\sC\to \Coh^{\Ce^\times}(\fB)$.
 Here the action of $\Ce^\times$ on $\fB$ comes from the action on $\check{\fc}$, $\check{\fh}$
 given by $t:x\mapsto t^2x$.
\end{Thm}

The proof occupies most of this section. Before presenting it
we  spell out some properties of the functor
$R\Gamma_\LK$. 

\begin{Rem}
In the case when $G_\RE$ is a complex group the Theorem reduces to a special
case of  \cite[Proposition 3.1.6]{BY} which is a generalization of \cite[Proposition 3.4.2]{BGS}.
A geometric approach to the latter also yielding a proof of the former has been 
developed in \cite{Gin}.
We did not see a way to prove Theorem \ref{RGamma_full} in a similar fashion and opted
for a different strategy.
\end{Rem}

\subsection{Components of $\fB$ and closed orbits}
\label{sect_comp}
The irreducible components of $\fB$ are indexed by 
$W/W(\LK)$ and can be described as follows. Let us recall that we have fixed a particular bijection $\phi:(\LK\bs \LBB)_{cl}\to \LW^\Lth/W(\LK)$ in~\ref{21}. In particular, we have fixed a  closed $\LK$-orbit $S_0$ corresponding to the identity coset. Observe that in this manner the bijection $\phi$ extends to a bijection between the $\LK$-orbits on $\LBB$ attached to the torus $\check T$ and $\LW/W(\LK)$. Choosing a Borel given by a fixed point of $\check T$ on the orbit $S_0$ we obtain an identification $\psi: \Lt\xrightarrow{\sim}\Lh$ which is well defined up to conjugacy by $W(\LK)$. 

From the identification $\psi: \Lt\xrightarrow{\sim}\Lh$ we obtain an embedding
 $\iota: \Lc\to \Lh$ which is well-defined up to conjugation by $W(\LK)$. For an element $w\in W$ 
 the map $\gamma_w:\Lc\to \fB$, $\gamma_w:x\mapsto (x,w\circ \iota)$
 is obviously a closed embedding and its image  $\fB_w$ is an irreducible component of $\fB$. 
  For $w'\in w \cdot W(\LK)$ we have 
  $\fB_w\cong \fB_{w'}$, thus we obtain a well-defined component $\fB_w\subset \fB$ for any
     $w\in \check W/W(\LK)$. Note that in this manner we have also obtained a particular bijection between the irreducible components of the  block variety $\fB$ and the set of $\LK$-orbits on $\LBB$ attached to the torus $\check T$.

\begin{Lem}\label{Ow}
For $S\in (\LK\bs \LBB)_{cl}$ we have 
$$R\Gamma_{\check K}(\Ce_S)\cong \O_{\fB_{\phi(S)}},$$
where $\phi$ is as in Claim \ref{21}.
\end{Lem}

\proof 
 As $S$ is isomorphic to the flag variety for $\LK$, we see that $H^*_{\LK}(S)\cong
\bC[\Lc]$. By choosing a $\check T$-fixed point on $S$ we obtain a $\theta$-stable Borel and that Borel gives rise to 
an identification $\psi_S: \Lt\xrightarrow{\sim}\Lh$. This identification is related to the identification $\phi$ by the element $w=\phi(S)\in W^\Lth$. Thus, the map 
$$
\bC[\Lh]\xrightarrow{\psi_S^*}\bC[\Lt]=H^*_\LG(\LBB)\to H^*_\LK(S)=\bC[\Lc]
$$ 
gives rise to  the map $w\circ \iota:\Lc\to \Lh$. This gives us the conclusion. 
  \qed
 
 \subsection{Convolution with $\alpha$-lines}
 \label{ZFSF}

Recall that for every simple root $\alpha$ we have the corresponding minimal
parabolic $\check P_\alpha$, the partial flag variety $\check\PP_\alpha$ and a $\Pone$ fibration
$\pi_\al: \LBB \to \check\PP_\al$.  Set $C_\al=\pi_\al^*\pi_{\al*}:\LD\to \LD$.
By the decomposition theorem $C_\al$ acts  also on the category $\sC$.

Set $\fB_\al=\Lc/W(\LK)\times _{\Lh/W} \Lh/W_\al$ where $W_\al=\{1,s_\al\}\subset W$
with $s_\al$  the simple reflection of type $\al$.
We have a degree two map $pr_\al: \fB \to \fB_\al$. 

\begin{Lem}\label{C_al}
We have a canonical isomorphism of functors: $\LD \to \Coh^{\Ce^\times}(\fB)$:
$$R\Gamma_{\LK}\circ C_\al\cong pr_\al^*pr_{\al*} \circ R\Gamma_{\LK}.$$
\end{Lem}

\begin{proof} We have a map 
$$
R\Gamma_{\LK}(\F)=R\Gamma_{\LK}(\pi_{\al*}(\F))
\to R\Gamma_{\LK}(\pi_{\al*}\pi_\al^*\pi_{\al*}(\F))=R\Gamma_{\LK}(\pi_\al^*\pi_{\al*}(\F))\,.
$$
 It is clear
that this map is compatible with the action 
of $H^*_{\LK}(\check\PP_\al)=\O(\fB_\al)$. Thus we get a map from the right hand side to the left hand 
side of the required isomorphism. 
The direct image of the constant sheaf $\pi_{\al *}(\Ce_{\LBB})$ is isomorphic to 
$\Ce_{\LBB_\al}\oplus \Ce_{\LBB_\al}[-2]$. It follows that 
 $$R\Gamma_{\LK}(\pi_\al^*\pi_{\al*}(\F))\cong R\Gamma_{\LK}(\pi_{\al*}(\F))
 \oplus R\Gamma_{\LK}(\pi_{\al*}(\F)) [-2]\cong R\Gamma_{\LK}(\F)\oplus
 R\Gamma_{\LK}(\F)[-2],$$
which implies that map in the lemma is an isomorphism. 
\end{proof}

\subsection{Closed orbits generate the block}
The following statement follows from the structure of the block as described
in \cite{ic4}, see Remark \ref{prop_from_Vogan}. We present a geometric argument for completeness.

\begin{Prop}\label{spec_gen} The objects $\underline{\Ce}_S[\dim S]$, $S\in (\LK\bs \LBB)_{cl}$  generate $\sC$
under the action of the functors $C_\al$, direct sums, summands and homological shifts; here  $\underline{\Ce}_S$
is the constant sheaf on a closed orbit $S$
\end{Prop}
\begin{proof}
Introduce a partial order on the set of irreducible $\LK$-equivariant perverse sheaves on $\LBB$
by saying that $L<L'$ if $\dim(\supp(L))< \dim (\supp(L'))$ and $L'$ is a direct summand
in the convolution $L*C$ for some semisimple complex $C\in D_\LB(\LBB)$. 
We will say that an irreducible perverse sheaf is {\em minimal} if it is a minimal element
for that partial order. 
The constant sheaf on a closed orbit is clearly minimal.  
It suffices to show that if $L$ is minimal but is not of this form then we have: 
$$
Ext^\bu(\underline{\Ce}_S*A_1, L*A_2)=0
$$
for $A_1,A_2\in D_\LB(\LBB)$. Using adjunction we reduce
to $A_2$ being the unit in the monoidal category $D_\LB(\BB)$. Then without 
loss of generality we can take $A_1$ to be an element in a generating set of the
triangulated category $D_\LB(\LBB)$.  Thus we are reduced to showing that

\begin{equation}\label{Ext_van}
Ext^\bu(\underline{\Ce}_S*J_{w!}, L)=0
\end{equation}
where $J_{w!}$ is the constant sheaf on the Schubert cell associated to $w$ extended by zero. 
Let $O$ be the open orbit in the support of $L$ and $\cL$ the corresponding 
local system on $O$. Let $B_x$ be the stabilizer of a point $x\in O$,  and let $T_x\subset B_x$
be a $\theta$-stable Cartan. 

Since $O$ is not closed, $B_x$ is not $\theta$-stable, so there exists a simple negative
root $\al$ for $T_x$ such that $\theta(\al)$ is positive. We claim that in this case
$\theta(\al)=-\al$ (i.e. $\al$ is a real root): otherwise the corresponding $\Pone$ fibration
$\pi_\al:\LBB\to \LPP_\al$ restricted to $O$ is an $\Aone$ fibration, so 
$L$ is a direct summand of $\pi_\al^*\pi_{\al*}(L')[1]$ for a simple perverse sheaf $L'$
whose support is the closure of $\pi_\al^{-1}\pi_{\al}(O)\setminus O$. This
contradicts the minimality of $L$. 

For each $\al$ with $\theta(\al)=-\al$ the fibers of $\pi_\al|_{O}$ are isomorphic
to $\Ce^*$; moreover, the local system $\cL$ is nontrivial on such a fiber, for otherwise
one can find a simple perverse sheaf $L'$ whose support is the closure of either  $\pi_\al^{-1}\pi_\al(O)\setminus O$
or one of its two components,  such that $L$ is a direct summand in $\pi_\al^*\pi_{\al*}(L')[1]$. 
 
 By a result of Speh and Vogan \cite{SpV}, see also \cite[Theorem 8.6.6]{green} and \cite[Corollary 8.8.]{HMSW},
 $L$ is a clean extension of $\cL[\dim (O)]$ from $O$, thus the left hand side
 of \eqref{Ext_van} can be rewritten as 
 $H^\bu_{\LK}(Z(w),pr_2^*(\cL))$ where $Z(w)=(S\times O)\cap \LBB^2_w$ and $pr_2$ stands
 for the second projection (here $\LBB^2_w$ is the $\LG$-orbit in $\LBB^2$ corresponding to $w$). 
 Fix $y\in S$ and let $Z_y(w)=Z(w)\cap pr_1^{-1}(y)$, then  \eqref{Ext_van} would follow if we show  
 that  
 \begin{equation}\label{LZy}
 H^\bu(Z_y(w),\cL|_{Z_y(w)})=0.
 \end{equation} 
 
We will prove \eqref{LZy} by considering intersections of $Z_y(w)$ with $\alpha$-lines,
i.e. fibers of $\pi_\al$ for a simple root $\al$. To this end  we first check that for every $w$ such that
 $Z(w)$ is nonempty there exists  an $\al$ with  $\theta(\al)=-\al$ and such that $ws_\al<w$.
 
 Assume that $x\in Z_y(w)$. Let $P$ be the parabolic generated by $B_x$ and 
 the root subgroups corresponding to negative simple roots $\al$ with $\theta(\al)=-\al$,
 thus $P$ is $\theta$-stable parabolic. Therefore the image of $B_y\cap P$ in the Levi quotient
 $L_P$ is a $\theta$-stable Borel in $L_P$. Hence, it cannot coincide with the image of $B_x$
 in $L_P$ since the image of $B_x$ is in a general position with its image under $\theta$. Thus the image
 of $Lie(B_y\cap P)$ in $Lie(P)/Lie(B_x)$ contains the image of a simple
 negative root space, which by the definition of $P$ satisfies $\theta(\al)=-\al$.
 It follows that $ws_\al<w$. 
 
 For $\al$ as above consider $\tilde{Z}_y(w)=\pi_\al^{-1}\pi_\al(Z_y(w))\cap O$. 
 We have $\tilde{Z}_y(w)=Z_y(w)\cup Z_y(ws_\al)$ (note that $Z_y(ws_\al)$ may
 be empty). Using induction on the length of $w$ we can assume that \eqref{LZy} holds
 for $ws_\al$ (it obviously holds if $Z_y(w)$ is empty). Then \eqref{LZy} follows from:
  \begin{equation}\label{LZytil}
 H^\bu(\tilde Z_y(w),\cL|_{\tilde Z_y(w)})=0.
  \end{equation}
  It remains to observe that $\pi_\al |_{\tilde{Z}_y(w)}$ is a $\Ce^*$ fibration,
while the restriction of $\cL$ to a fiber of that fibration is nontrivial. It follows
that $R (\pi_\al |_{\tilde{Z}_y(w)})(\cL)=0$ which implies \eqref{LZytil} and hence
\eqref{LZy}.
 \qed
 
 \begin{Rem}\label{prop_from_Vogan}
 It is easy to see that the bijection induced by Vogan's character duality \cite{ic4}
   between irreducible objects in the two dual sides  interchanges irreducibles minimal
  in the 
 above sense with IC extensions from the open orbit. As observed above,
 IC extensions from the open orbit in $\MM$ are in bijection with 
 closed orbits in $\LBB$, this yields another proof of Proposition \ref{spec_gen}.
 \end{Rem}
 
  \begin{Rem} Similar considerations (see also
  also ~\cite[Theorem 8.5]{ic4})  show that
  given any block there is a $\theta$-stable Cartan $T$ well defined up to conjugacy
  with the following property. The minimal elements in the block are local systems on 
  orbits attached to $T$.  
  Moreover, if an object in this block restricted
  to an orbit $O\subset \LBB$ is nonzero then the corresponding $\theta$-stable 
  Cartan is less compact than $T$, i.e.\footnote{There is also a stronger
  sense in which one Cartan subgroup is more compact than the other,
  see {\em loc. cit.}.} 
   dimension of the space of $\theta$-invariants is 
  smaller. 
  \end{Rem}
  
\end{proof}

We now embark on the proof of 
Theorem \ref{RGamma_full}.

\subsection{Beginning of the proof of Theorem \ref{RGamma_full}: reduction
to closed orbits}
 We need to show that 
\begin{equation}
\label{isom}
Hom_\sO(\F,\G) \iso Hom_{\Coh(\fB)} (R\Gamma_\LK(\F), R\Gamma_\LK(\G))
\end{equation}
for $\F,\G \in \sO$, where the isomorphism is compatible with the grading. This implies both the graded and non-graded statements
of Theorem~ \ref{RGamma_full}. 
Notice that the functor $C_\al$ is self-adjoint (up to a shift) and $pr_\al^*pr_{\al*}$ is self-adjoint
(up to a degree shift if one considers graded modules).
Thus, the validity of isomorphism \eqref{isom} for a given $\G=\G_0$ and all $\F$ implies
its validity for $\G=C_\al(\G_0)$ and all $\F$. In view of Proposition \ref{spec_gen}
we conclude that it suffices to check \eqref{isom} for $\G={\Ce}_S[\dim S]$, $S\in (\LK\bs \LBB)_{cl}$.

\subsection{Pointwise purity and reduction to a property of the map on cohomology}
 Fix $\F\in \sO$, $S\in (\LK\bs \LBB)_{cl}$ and let $i_S:S\imbed \LBB$ be the embedding.
Then, by adjunction, 
\begin{equation*}
Hom_\sO(\F,\Ce_S)= Ext_{D_{\LK}(S)}(i_S^*(\F), \Ce_S).
\end{equation*}

We have
\begin{equation}
\label{semi-simple}
\text{the complex $i_S^*(\F)$ is semi-simple}\,.
\end{equation}
To see this, one has to work in the category of mixed sheaves. This can be done directly in this setting by utilizing Saito's mixed Hodge modules \cite{S1,S2} or, by passing to the case of a finite base field and using $\ell$-adic sheaves. In either case, the sheaf $\F$ is pure. It follows from the arguments in \cite{LV} that $\F$ is actually {\em pointwise pure}, i.e., for all $x\in X$ the sheaves $i_x^*\cF$ and  $i_x^!\cF$ are pure; here $i_x: \{x\} \to X$ denotes the inclusion. As $S$ is a closed orbit of $\LK$ on $\LBB$ it is isomorphic to the flag manifold of $\LK$ and hence is simply connected. Therefore, 
the complex 
 $i_S^*(\F)$ has constant cohomology sheaves, and its stalks are pure because of 
 pointwise purity of $\F$. Thus $i_S^*(\F)$  
 is a pure complex and hence it is semi-simple.
 As $i_S^*(\F)$ is semi-simple, it is a direct sum of sheaves $\Ce_S$ that have been shifted to various degrees. From this we conclude immediately that 
 
 \begin{equation*}
Ext_{D_{\LK}(S)}(i_S^*(\F), \Ce_S) = Hom_{H^*_{\LK}(S) } (R\Gamma_{\LK}(i_S^*(\F)) , H^*_{\LK}(S))\,,
\end{equation*}
 and hence that 
\begin{equation}
\label{reduction}
 Hom_\sO(\F,\Ce_S)= Hom_{H^*_{\LK}(S) } (R\Gamma_{\LK}(i_S^*(\F)) , H^*_{\LK}(S)).
 \end{equation}
 
Recall now that we are attempting to show that the canonical map
\begin{equation*}
Hom_\sO(\F,\Ce_S) \to Hom_{\Coh(\fB)} (R\Gamma_\LK(\F), H^*_{\LK}(S))
\end{equation*}
is an isomorphism. Hence, by \eqref{reduction}, it suffices to show that the canonical map 
\begin{equation}
\label{reduction1}
 Hom_{H^*_{\LK}(S) } (R\Gamma_{\LK}(i_S^*(\F)) , H^*_{\LK}(S))\to Hom_{\Coh(\fB)} (R\Gamma_\LK(\F), H^*_{\LK}(S))\,,
 \end{equation}
 which is induced by the map $R\Gamma_{\LK}(\F) \to R\Gamma_\LK(i_S^*(\F))$ on global cohomology, 
 is an isomorphism. Let us now write down this map in more concrete terms. 
 
 It is clear from Lemma \ref{Ow} that
  the map in \eqref{reduction1} can be written as
 \begin{equation*}
Hom_{\Coh(\fB_w)}  (R\Gamma_{\LK}(i_S^*(\F)) , \cO_{\fB_w}))\to Hom_{\Coh(\fB)} (R\Gamma_\LK(\F), \cO_{\fB_w})\,,
\end{equation*}
where $w=\phi(S)$ using the notations of Claim \ref{21} and Lemma \ref{Ow}. 

Using adjunction on the right hand side, we are reduced to showing that 
\begin{equation}
\label{last step}
Hom_{\Coh(\Lc)}  (R\Gamma_{\LK}(i_S^*(\F)) , \cO_\Lc))\to Hom_{\Coh(\Lc)} (\cO_\Lc \otimes_{\cO_\fB}R\Gamma_\LK(\F), \cO_\Lc)
\end{equation}
is an isomorphism. 

Finally, let us recall that the map in the above formula is induced by the canonical map of coherent $\cO_\Lc$-sheaves
 \begin{equation}\label{HtoH}
\cO_\Lc \otimes_{\cO_\fB}R\Gamma_\LK(\F) \xrightarrow{\ r\ } R\Gamma_\LK(i_S^*(\F)).
\end{equation}
We  observe that the map  \eqref{last step} is obtained from  \eqref{HtoH} by duality, i.e., by applying the functor $Hom(\ ,  \cO_\Lc)$. Thus, to prove that~\eqref{last step} is an isomorphism it suffices to show that
\begin{align*}
& Hom_{\Coh(\Lc)} (Ker(r), \cO_\Lc))\ = \ 0
\\
&  Hom_{\Coh(\Lc)} (Coker(r), \cO_\Lc))\ =\  Ext^1_{\Coh(\Lc)} (Coker(r), \cO_\Lc))\ = \ 0 
 \end{align*}

Thus, we conclude that the map in \eqref{last step} is an isomorphism if
\begin{equation}
\label{isomorphism reduction}
\begin{gathered}
\text{a. the support of the kernel of \eqref{HtoH}  has positive codimension and}
\\
\text{b. the support of the cokernel of \eqref{HtoH} has codimension two 
or higher.}
\end{gathered}
\end{equation}
As coherent $\cO_\Lc$-sheaves are generically locally free, it suffices to check the statement above on the level of fibers. 
\subsection{Small codimension  argument and localization to fixed points}

We will now argue~\eqref{isomorphism reduction}. To do so we will make use of localization of equivariant cohomology for torus actions. Thus, we will shift to equivariant cohomology with respect to the torus $\check C$. We recall that, in general, 
\begin{equation*}
H^*_\LK(X,\F) = H^*_{\check C}(X,\F)^{W(\LK)}\,;
\end{equation*}
here $X$ is a variety and $\F$ is a $\LK$-equivariant sheaf in $D_\LK(X)$.

Thus, to prove \eqref{isomorphism reduction} for the map \eqref{HtoH} it first of all suffices to prove it for the map
 \begin{equation}
 \label{final reduction}
 H^*_{\check C}(S)\otimes_ {H^*_{\check C}(\LBB)}H^*_{\check C}(\LBB,\F) \to  H^*_{\check C}(S,i_S^*(\F))
\end{equation}
and it suffices to check the statement on the level of fibers. 

To check the statement on the level of fibers, let us also recall the localization theorem for equivariant cohomology, see, for example, \cite[Theorem 6.2]{GKM}. We view $H^*_{\check C}(X,\F)$ as an  $H^*_{\check C}(\on{pt},\Ce)= \Ce[\check c]$ module. Let $c\in \check \fc$ be generic. In the case at hand this amounts to $c$ not lying on any root hyperplanes. Then we have:
\begin{equation*}
\cO_{\check\fc,c}\otimes_{\cO_{\check\fc}}H^*_{\check C}(X,\F) \cong \cO_{\check\fc,c}\otimes_{\cO_{\check\fc}}H^*_{\check C}(X^{c},\F)\,,
\end{equation*}
where $\cO_{\check\fc,c}$ stands for the local ring at $c$. Furthermore, taking the quotient by the maximal ideal we get an isomorphism on the fibers:
\begin{equation*}
H^*_{\check C}(X,\F)_c \cong H^*_{\check C}(X^{c},\F)_c\,,
\end{equation*}
We will apply the localization theorem in this form. 

Let us recall that $Z_{\check \fg}(\check \fc)=\check \ft$ is a maximal torus in $\check \fg$. 
Recall also that the torus $\check \fc$ has elements which are regular in $\check \fg$, i.e., elements that  do not lie in any root hyperplane for $\fg$. Let us first consider such a $c\in \check \fc$. 

Then by the above localization theorem for equivariant cohomology we get isomorphisms of algebras and restrictions as follows:
\begin{equation}
\label{localization of algebras}
\begin{CD}
H^*_{\LC}(\LBB)_c @>{\cong}>>  \oplusl_{x\in \LBB^c}
 \Ce_x
 \\
 @VVV   @VVV 
 \\
 H^*_{\LC}(S)_c@>{\cong}>>  \oplusl_{x\in S^c}  \Ce_x\,.
\end{CD}
\end{equation}
Furthermore, we get:
\begin{equation}
\label{localization of F}
\begin{CD}
R\Gamma_{\LC}(\F)_c @>{\cong}>>  \oplusl_{x\in \LBB^c}
 \F_x
 \\
 @VVV   @VVV 
 \\
R\Gamma_{\LC}(i_S^*\F)_c@>{\cong}>>  \oplusl_{x\in S^c}  \F_x\,.
\end{CD}
\end{equation}
In the above formulas $\F_x$ stands for the stalk of the complex $\F$ as the point $x$. The isomorphisms and restrictions in \eqref{localization of F} are compatible with the actions of the algebras in \eqref{localization of algebras}.  Thus, we conclude that for $c$ regular in $\g$ we have
\begin{equation*}
( H^*_{\check C}(S)\otimes_ {H^*_{\check C}(\LBB)}H^*_{\check C}(\LBB,\F))_c \cong  H^*_{\check C}(S,i_S^*(\F))_c.
\end{equation*}
Thus we established part a) of~\eqref{isomorphism reduction} and that the cokernel in part b) of~\eqref{isomorphism reduction} is concentrated on root hyperplanes. It remains to prove that the cokernel vanishes generically on root hyperplanes. 
 
 \subsection{The small rank case}
Let us then assume that $c\in \Lc\cap H_\al$ is a general element; here $H_\al\subset \Lh$ is the root hyperplane corresponding to the root $\alpha$  of $\Lh$ in $\Lg$. We write $\check C_\alpha$ for the codimension one sub torus of $\check C$ corresponding to the root $\alpha$. In other words, $\check C_\alpha$ is the torus generated by $c$.

Let us write  $Z_c'$  for the  derived group of the centralizer of $c$ in $\LG$. Then,  by construction, the involution ${\check\theta}$ restricts to $Z_c'$  and the group
$(Z_c')^{\check\theta}$ of fixed points of ${\check\theta}$ has rank 1. This imposes a very severe restriction on $Z_c'$. It can only happen
  if $Z_c'$ has rank one or if it has rank two and the Cartan involution is not inner.
  Thus $Z_c'$ corresponds (up to isogeny) to one of the following real groups:
  $SL(2,\RE)$, $SU(2)$, $SL(2,\Ce)$, $SL(3,\RE)$. It is also easy to see this from the geometric point of view as follows. Because $(Z_c')^{\check\theta}$  has  rank one and it has to have an open orbit on the flag manifold of $Z_c'$ it follows that the flag manifold can have at most dimension three. Thus, we recover the list above. Furthermore, as we have assumed that $\LK$ is connected, it follows that $(Z_c')^{\check\theta}$ is connected.
  
  The fixed point set $\LBB^c$ is a union of connected components
  each one of which is isomorphic to the flag variety of $Z_c'$. It follows from the discussion above that the components of $\LBB^c$ are either $\Pone$ or $\Pone \times \Pone$ or the quadric $Q\subset {\mathbb P}^2\times {\mathbb P}^2$. 
  
  Let us recall that we have to argue that the map in~\eqref{final reduction} is surjective in codimension one, i.e., that for the $c\in \check C_\alpha$  generic we have surjection:
 \begin{equation*}
(H^*_{\check C}(S)\otimes_ {H^*_{\check C}(\LBB)}H^*_{\check C}(\LBB,\F))_c \to H^*_{\check C}(S,i_S^*(\F))_c\,.
\end{equation*}
It suffices to prove that $H^*_{\check C}(\LBB,\F)_c \to H^*_{\check C}(S,i_S^*(\F))_c$ is a surjection.

By Proposition \ref{spec_gen} 
we can realize $\F$ as a direct summand in $C_{\alpha_1}  \cdots C_{\alpha_n}(\underline{\Ce}_O[d_O])$,
 for a collection of simple roots
$\alpha_1,\dots, \alpha_n$ and a closed $\LK$-orbit $O$ of dimension $d_O$. Let ${\bf C}$ denote the corresponding convolution diagram 
and $\pi:{\bf C}\to \LBB$ the convolution morphism. 

Then $\F$ is a direct summand in $\pi_*(\underline{\Ce}_{\bf C}[\dim ({\bf C})])$, applying base change to the proper 
morphism $\pi$ we are reduced to showing that the restriction 
map $H^*_{\check C}({\bf C})_c \to  H^*_{\check C}(\pi^{-1}(S))_c$ is onto. 

Applying again the localization theorem for equivariant cohomology~\cite[Theorem 6.2]{GKM}, we are reduced to 
showing surjectivity of the restriction map for the fixed point sets: $H^*({\bf C}^c)\to  H^*(\pi^{-1}(S)^c)$. Finally, setting
 $\F^c:=\pi_*(\underline{\Ce}_{{\bf C}^c})$ we get that this is equivalent to surjectivity 
 of the map   $H^*(\LBB^c,\F^c)\to H^*(S^c,i_S^*(\F^c))$.
 
\begin{lem}
\label{purity of restriction}
The complex $\F^c$ is semi-simple. 
\end{lem}

\begin{proof}
The space ${\bf C}^c$ is the fixed point set of an algebraic torus acting on the smooth space $\bf C$, thus
it is smooth. So the claim follows from the Decomposition Theorem \cite{BBD}. 
\end{proof}

We conclude that it suffices to prove  surjectivity of restriction for an irreducible $\check K\cap  Z_c'$-equivariant sheaf $\G$ on  $Y$, i.e., that 
 \begin{equation}
 \label{surj}
H^*(Y,\G) \to H^*(S\cap Y,i_S^*(\G))
\end{equation}
is surjective, where $Y$ is a component of $\LBB^c$. Notice that $Y$ is identified with the flag manifold of $Z_c'$. As $S$ is a closed orbit it consists of ${\check\theta}$-stable Borels in $\LBB$,therefore $S\cap Y$ consists of ${\check\theta}$-stable Borels in $Y$.

Recall that we have the following possibilities. When $Y=\bP^1$ then $(Z_c')^{\check\theta}$ is, up to isogeny, $SL(2)$ or $\bC^*$. When $Y=\bP^1\times \bP^1$  then $(Z_c')^{\check\theta}$ is, up to isogeny, $SL(2)$ and when $Y=Q$ then $Z_c'$ is, up to isogeny, $SO(3)$. 

In all these cases $Y^{\check\theta}$ is a single $(Z_c')^{\check\theta}$ orbit with the exception of $Y=\bP^1$, 
$(Z_c')^{\check\theta}=\bC^*$ when it consists of two points. We claim that in the latter case $|S\cap Y|$=1. To see this, let
 $Y^{\check\theta}=\{x,y\}$ and $B_x=T_xU_x$ be the stabilizer of $x$,
then the line $Y$ corresponds to a non-compact imaginary root $\alpha$ of $T_x$.
It follows that $\alpha$ is not a sum of two $\theta$-conjugate complex roots (for if $\alpha=\alpha_1+\alpha_2$ with $\theta(\alpha_1)=\alpha_2$ then $\langle c,\alpha_i\rangle =0$
for $i=1,2$ contradicting $Z_c'\cong SL(2)$), which implies that $s_\alpha\not \in W_{\LK}$, so $y\not \in \LK(x)$.

 If $\G$ is an intersection cohomology complex associated to a non-constant local system then $Z_c'$ corresponds to  $SL(2,\RE)$ or $SL(3,\RE)$ and the support of $\G$ is all of $Y$. As we will argue below, in this case $i_S^*(\G)=0$ and  the surjectivity in~\eqref{surj} obvious. When  $Z_c'$ corresponds to  $SL(2,\RE)$ one sees immediately that $i_S^*(\G)=0$ for the non-trivial  rank one local system. When $Z_c'$ corresponds to  $SL(3,\RE)$ there are three non-trivial local systems. When the rank one local system is non-trivial only around one of the codimension one orbits the smoothness of the closures of the codimension one orbits immediately implies that $i_S^*(\G)=0$. When the local system is non-trivial around both of the codimension one orbits a short calculation shows that $i_S^*(\G)=0$, or, alternatively, it also follows from the fact that $\G$ corresponds to the irreducible principal series and in this case $\G$ is clean. 

We are now reduced to consider the case when $\G$ is an intersection homology complex associated to a trivial local system. The smoothness of closures of the orbits then imply that $\G$  is a constant sheaf on a closure of an orbit. Let us write $R$ for an orbit closure. Considering the composition
 \begin{equation*}
H^*(Y,\bC) \to H^*(R,\bC) \to H^*(S\cap Y, \bC)
\end{equation*}
 we are reduced to showing that
 \begin{equation*}
H^*(Y,\bC) \to H^*(S\cap Y, \bC)
\end{equation*}
 is surjective. This follows from:
 \begin{prop}
 Let $G$ be a semisimple algebraic group with an involution $\theta$ such that the fixed point set $K$ of $\theta$ is connected. Let us write $Y$ for the flag manifold of $G$, $S$ for a closed $K$-orbit on $Y$ and $C$ for a Cartan of $K$. Then we have a surjection
$$
H^*(Y,\bC) \to H^*(S, \bC)\,.
$$
 \end{prop}

\begin{proof}
The closed orbit $S$ is a flag manifold of $K$ and we can think of the map $S\subset Y$ as an inclusion $K/K\cap B \subset G/B$ for a particular Borel $B$ in $G$.
This yields an inclusion of the abstract Cartan $\fc$ of $K$ into $\ft$, the abstract Cartan of $G$.
We have onto maps $\Ce[\ft]\to H^*(Y,\Ce)$, $\Ce[\fc]\to H^*(S,\Ce)$ compatible with the inclusion $\fc \imbed \ft$, which implies the statement.
\end{proof}

This completes the proof of Theorem \ref{RGamma_full}.
   
  \subsection{The case of disconnected $\LK$}
 Let us now consider the general case, i.e., the case when $\LK$ is possibly not connected. We define the block variety $\fB$ using the group $\LK^\circ$, i.e., 
 \begin{equation*}
 \fB=\Lc/W(\LK^\circ,\Lc)\times _{\Lh/W} \Lh
 \end{equation*}
We note that we have the following exact sequence
 \begin{equation*}
 1 \to W(\LK^\circ, \Lc) \to W(\LK,\Lc) \to \fS \to 1\,,
 \end{equation*}
 where  $\fS = \LK/\LK^\circ$.
 Furthermore,  recall again that  $Z_{\check \fg}(\check \fc)=\check \ft$ is a maximal torus in $\check \fg$ and that the torus $\check \fc$ contains elements which are regular in $\check \fg$, i.e., elements that  do not lie in any root hyperplane for $\fg$. Thus, $N_\LG(\Lt) = N_\LG(\Lc)$. This implies that $W(\LK,\Lc)\subset W(\LG,\Lc) \subset W(\LG,\Lt)= W$. Thus: 
 \begin{equation*}
 \begin{gathered}
 \text{We have an action of $\fS$ on $\Lc/W(\LK^\circ,\Lc)$ such that}
\\
 \text{the induced
 action on the quotient $\Lc/W(\LG,\Lc)$ is trivial}\,.
 \end{gathered}
 \end{equation*}
 Therefore the action of $\fS$ on $\fB$ takes place on the factor $\Lc/W(\LK^\circ,\Lc)$. 
  
 We define a functor from $\LD$ to $\Coh^\fS(\fB)$ by sending $\cF\in \LD$ to  $R\Gamma_{\LK^\circ}(\cF)$. The case of connected $\LK$ immediately implies the general case:
 
 \begin{Thm}\label{RGamma_full general}
The functor $R\Gamma_{\LK^\circ}$ induces full embeddings $\sO \to \Coh^\fS (\fB)$, $\sC\to \Coh^{\fS\times \C^\times}(\fB)$.
\end{Thm}

\subsection{The description of the principal block}
The following Theorem summarizes the results of this section. 
To state it recall that the components of $\fB$ are in canonical
 bijection with $W/W_\MM'$. Let us call components corresponding to elements in the 
subset $W^\theta/W_\MM'$ {\em preferred} components. These are in bijection
with the set of closed $\LK^0$ orbits on $\LB$ and also with the set of $\LL_\MM^{ad}(X_0)$.

\begin{Thm}\label{pribl_thm}
Let $\cA\subset \Coh(\fB)$, $\cA_{gr}\subset \Coh^{\Ce^\times}(\fB)$
 be the full subcategories generated by the structure
sheaves of preferred components of $\fB$ under the action of the functors $pr_\al^*pr_{\al*}$,
direct sums and direct summands (respectively, direct sums, summands and shifts).
 Let $\cA^\fS$ be the full subcategory
in $\Coh^\fS(\fB)$ defined as the preimage of $\cA$ under the forgetful functor
$\Coh^\fS(\fB)\to \Coh(\fB)$ and similarly for $\cA^\fS_{gr}$.

Then the functor $R\Gamma_{\LK^0}$ induces  equivalences 
 $$\sO \iso \cA^\fS,$$
 $$\sC\iso \cA^\fS_{gr},$$
where $\sO$, $\sO^{gr}$ were defined in  section \ref{sect_full}.
\end{Thm}

\proof The second equivalence follows from the first one provided that
the isomorphisms on Hom's induced by the equivalence are compatible
with the gradings on the two sides. We proceed to construct such an equivalence.

Assume first that $\LK$ is connected.
Then the functor $R\Gamma_\LK|_{\cL}$ is fully faithful by Theorem
 \ref{RGamma_full}. Its image is generated under the functors $pr_\al^*pr_{\al*}$
 by the images of constant sheaves on closed orbits in view of Lemma \ref{C_al}
 and Proposition \ref{spec_gen}. These images are exactly the structure sheaves
 of components of $\fB$ by  Lemma \ref{Ow}. This proves the Theorem for connected
 $\LK$. 

Turning to the general case, we see that the functor is fully faithful
by Theorem \ref{RGamma_full general}. We claim that an object $\F\in \Coh^{\fS}(\fB)$
lies in the image of $R\Gamma_{\LK}|_{\cL}$ iff $\operatorname{Forg}(\F)\in \Coh(\fB)$ lies in the image of $R\Gamma
_{\LK^\circ}|_{\cL^\circ}$, where the notation is self-explanatory. 
The "only if" part of the statement is clear from the commutative diagram
$$\begin{CD}
\cL_\LK @>{Res^\LK_{\LK^\circ}}>> \cL_\LK^\circ \\
@V{R\Gamma_\LK}VV @VV{R\Gamma_{\LK^\circ}}V \\
\Coh^\fS(\fB) @>\operatorname{Forg}>> \Coh (\fB) 
\end{CD}
$$
The "if" part follows by considering adjoint to functors in the last diagram and observing 
that $\F\in \cL_K$ is a direct summand in $Ind_{\LK^\circ}^\LK \circ Res^{\LK}_{\LK^\circ}(\F)$.
\qed

\section{Integral regular blocks for a quasi-split group}
\label{sect2}
\subsection{Statement of results}
\label{statement}
In this section we provide a description of a block $\M$ in the category
of $(\g,K)$ modules with a regular integral generalized infinitesimal character
for a pair $(\g,K)$ coming from a quasi-split real group $\GR$. 

 It will be convenient to enlarge the category $\M$
 to the category $\Mh$ of pro-objects in $\M$. Every irreducible object in $\M$
 admits a projective cover in $\Mh$ which is unique up to an isomorphism.
 Let $\P\subset \Mh$ be the full  subcategory
 formed by finite sums of projective
 covers of irreducible (equivalently, any) objects in $\M$. It is easy to 
 see that $\P$ can be also realized as a full subcategory in the category of 
 finitely generated modules over $U(\g)\otimes _{Z(\g)} Z(\g)_{\hat{\la}}$
 equipped with a compatible $K$-action, where $Z(\g)_{\hat{\la}}$
 is the completion of the center of the enveloping $Z(\g)$ at the maximal
 ideal corresponding to $\la$. 
 
 Let $\fBh$ denote the completion of $\fB$ at zero, by which we understand the spectrum
 of the completion of the coordinate ring $\O(\fB)$ at the corresponding maximal ideal.
  Define a  full subcategory  $\Ah\subset \Coh(\fBh)$
 as the image of the category $\A$  introduced in Theorem \ref{pribl_thm}
 under the completion functor $\Coh(\fB)\to \Coh(\fBh)$. 
 Likewise, $\Ah^\fS\subset \Coh^\fS(\fBh)$ is defined as the essential image of
 $\A^\fS$.
 
 The goal of this section is the following
 
 \begin{Thm}
 \label{monodromic main theorem}
 There exists a canonical equivalence $\P\cong \Ah^\fS$.
 \end{Thm}
 
The proof of this theorem follows from propositions~\ref{singbl_prop}, \ref{translff_prop}, and \ref{Pgen} which we state in this subsection.
The proof of these propositions occupies the rest of the section. 
 
 Along with $\M$ we will need to consider the singular category
 of Harish-Chandra modules $\M_{sing}$. To define it assume first that
 $G$ is simply connected and thus $K$ is connected. Then we have an irreducible
 $G$-module with highest weight $\lambda+\rho$ which can be used
 to define the translation functor $T_{\la\to -\rho}$ from the category of
 $(\g,K)$-modules with generalized infinitesimal character $\la$ to 
 the category of $(\g,K)$-modules with generalized infinitesimal character $-\rho$.
 We write $\M_{sing}$  for the Serre subcategory generated by the image of $\M$ under 
 $T_{\la\to -\rho}$.

 We now drop the assumption that $G$ is simply connected. 
 Let $G_{sc}$ be the  simply connected
 cover of $G$ and let $K_{sc}$ be the preimage of $K$ under the covering map
 $G_{sc}\to G$. Recall that the finite group $Z(G)\cap K$ acts on all modules in $\M$
by a fixed character, which we denote by $\chi$. Let $\chiti$ be the pull back of
$\chi$ to the group $Z(G_{sc})\cap K_{sc}$. 
The pull back functor
from $(\g,K)$-modules to $(\g,K_{sc})$-modules is fully faithful and $Z(G_{sc})\cap K_{sc}$
acts on the modules in the image by $\chiti$.

Set $\chiti '=\chiti  \cdot \chi_\la^{-1}$ where $\chi_\la$ is the character by which
 $Z(G_{sc})\cap K_{sc}$ acts on the irreducible $G_{sc}$ module of highest weight
 $V_{\la+\rho}$.
 Consider  the category of $(\g,K_{sc})$  modules
 with generalized infinitesimal character $-\rho$, where $Z(G_{sc})\cap K_{sc}$ acts
 by the character $\chiti'$. 
The  translation functor $T_{\la\to -\rho}$ sends $\M$ to that category; we let 
 $\M_{sing}$ denote the Serre subcategory generated by the image of $\M$ under 
 $T_{\la\to -\rho}$.

We write $\P_{sing}$ for the category of projective pro-objects in $\M_{sing}$
which are finite sums of projective covers of irreducible objects. Let us write $\Coh_0^\fS(\a^*/\Wbp)$ for the category of $\cS$-equivariant coherent sheaves on $\a^*/\Wbp$ set theoretically supported at zero. We denote by $\widehat{\a^*/\Wbp}$ the completion of $\a^*/\Wbp$ at zero and then write $\Coh_{fr}^\fS(\widehat{\a^*/\Wbp})$ for the category of
projective (equivalently, free) $\fS$-equivariant coherent sheaves on  $\widehat{\a^*/\Wbp}$.

\begin{Prop}\label{singbl_prop}
We have canonical equivalences: $\M_{sing}\cong \Coh_0^\fS(\a^*/\Wbp)$ and 
$\P_{sing}\cong \Coh_{fr}^\fS(\widehat{\a^*/\Wbp})$.
\end{Prop}

Recall the extended enveloping algebra  $\widetilde U(\fg)=U(\g) \otimes _{Z(\g)}
Sym(\h)$, where the action of the center of the enveloping algebra $Z(\g)$
on $Sym(\h)$ comes from the Harish-Chandra isomorphism, see, for example, 
 \cite{BeGi}.

Let $\widetilde \M_{sing}$ be the category of $(\widetilde U(\fg), K)$-modules, such that
restricting the action of $\widetilde U(\fg)$ to $U(\g)$ one gets a module in $\M_{sing}$.
We will call  $\widetilde \M_{sing}$ the extended singular block.

\begin{Cor}
\label{cor}
The extended singular block $\widetilde \M_{sing}$  is
naturally equivalent to $\Coh_0^\fS(\fB)$.
\end{Cor}

The arguments in  section 1 of \cite{BeGi} show that the translation functor $T=T_{\la\to -\rho}$
admits a natural lifting to the ``extended translation functor"
$\widetilde T:\M\to \widetilde \M_{sing}$. If we write $res: \widetilde \M_{sing} \to \M_{sing}$ for the restriction functor then 
\begin{equation}\label{Tres}
T = res \circ \widetilde T\,.
\end{equation}
By Corollary~\ref{cor} the extended translation functor can be viewed as a functor $\widetilde T:\M \to \Coh_0^\fS(\fB)$.
The functor $\widetilde T$ extends to the category of pro-objects and we continue to write $\widetilde T$ for the  resulting extension.
 
\begin{Prop}\label{translff_prop}
The functor $\Tt|_\P$ is  fully faithful.
\end{Prop}

Before describing the last ingredient of the proof we define an important set 
of objects in $\Mh$. Let $L=j_{!*}(\LL[\dim X_L])$  be an irreducible object in $\M$, where we have written $j:X_L\to X$ for the embedding  of
the open $K$-orbit in the support of $L$ and $\LL$ for the corresponding irreducible local system.
Let $\cE$ be the free pro-unipotent local system on the torus $A_L=X_L/K$
(see \cite{Beil} for the original reference and \cite{BY} where it was used in a closely related context).

\begin{Def}\label{def_pr}
The {\em deformed principal series} module $\Pi_\LL$ is given by:
$$\Pi_\LL=j_!(\LL\otimes \cE[\dim X_L]).$$
\end{Def}
We consider $\Pi_\LL$ as an object in $\Mh$. 
We will only be interested in  the deformed principal series modules attached to the maximally
split Cartan, i.e., in the cases when dimension of the torus $A_L$ is maximal.
It is easy to see that $\Pi_\LL$ is a projective cover of $L$ in the full
abelian subcategory of $\Mh$ consisting of sheaves supported on the closure of $X_L$.
In particular, $\Pi_\LL\in \P$ when $X_L$ is the open orbit.

\medskip
 
 Recall that the we have fixed a bijection between $W^\theta/\Wb$ and the
 set of $K\times H$-equivariant local systems on $X_0$ such that $j_{!*}(\LL)\in \MM$.
 For $v \in W^\theta/\Wb$ let $\LL_{v}$ be the corresponding local system and $L_v$
  be the corresponding irreducible object.

Our next goal is to describe  the image of the fully faithful functor $\Tt$ on $\cP$. To that end we recall the notion of wall-crossing functors, we
use \cite{BeGi} and \cite[Chapter 7]{Hum} as a general reference.  
Recall that $R_\al =  T_{\mu \to \la}\circ T_{\la\to\mu}$ where $\mu$ is a (generic) weight on the $\al$-wall. 
The translation functors $T_{\la\to\mu}$ are defined in the same manner as the $T_{\la\to-\rho}$ introduced earlier. 
It is a standard fact (see e.g. \cite[Proposition 3.1(a)]{BeGi}) that for an irreducible module $L$ its translation to the wall $ T_{\la\to\mu}(L)$ is either zero or irreducible. Furthermore, if we are given another irreducible module $L'\not \cong L$ we have $T_{\la\to\mu}(L)\not \cong T_{\la\to\mu}(L')$ unless $T_{\la\to\mu}(L)=0= T_{\la\to\mu}(L')$. Using adjointness between $T_{\mu \to \la}$ and $ T_{\la\to\mu}$
we see that for an irreducible module $L$ the module $R_\al(L)$ has a simple socle isomorphic to $L$, thus it belongs to 
the block of $L$ and we get the {\em wall-crossing functor}
$R_\al:\M\to \M$. We note further that $R_\al:\P\to \P$.

Also, recall from subsections \ref{sect_comp} and \ref{matching}
 that the components of $\fB$ are in bijection with
 $W/\Wbp$, let $\fB_v$ be the component corresponding to $v\in W^\theta/\Wbp$.
 The notation $W_\al=\{1,s_\al\}\subset W$ was introduced in subsection~\ref{ZFSF}.
 
\begin{Prop}\label{Pgen}
a) $\P$ is generated by the  deformed principal series modules
$\Pi_{\LL_v}$, $v\in W^\theta/\Wbp$ under the action of the functors $R_\al$ and taking direct
sums and summands.

b) We have a natural isomorphism 
$$\Tt \circ R_\al \cong pr_\al^* pr_{\al *}\circ \Tt.$$
Here we identified the target category of $\Tt$ with $\Coh_0(\fB)$, 
notation 
 $pr_\al$
has been introduced in section \ref{ZFSF}.

c) If $G$ is adjoint (so that $\Wb=\Wbp$) then $\Tt$ sends $\Pi_{\LL_v}$ to $\hatt{\O_{\fB_v}}$,
$v\in W^\theta/\Wbp$.
For a general $G$ the functor $\Tt$ sends $\Pi_{\LL_v}$ to $\bigoplus\limits_{\widetilde v} 
\hatt{\O_{\fB_{\widetilde v}}},$ where $\widetilde v$ runs over the set of elements in $W^\theta/\Wbp$
projecting to $v\in W^\theta/\Wb$; the sheaf 
$\bigoplus\limits_{\widetilde v} \hatt{\O_{\fB_{\widetilde v}}}$ is equipped with the obvious $\fS$-equivariant structure.

\end{Prop}

\proof 
 Recall that (see e.g. \cite{BeGi}), 
for $M\in \MM$ its wall-crossing 
$R_\alpha(M)$ 
fits in the distinguished triangle 
\begin{equation}\label{RtoI}
M\to R_\al(M)\to  I_\al(M)\to M[1],
\end{equation}
where  $I_\al$ is the  intertwining functor corresponding to the simple root $\al$, see section \ref{cross}.

 This allows us to employ the method of the geometric proof of Casselman's submodule theorem in \cite{BeBe}. Let us consider the projective cover $P_L$ of the irreducible object $L$ in $\MM$. There exists a sequence of simple roots $\al_1,\dots, \al_m$,
such that $R_{\al_1}\circ \cdots \circ R_{\al_m}(L)$ has nonzero restriction to
the open orbit in $X$. We see this as follows. Using induction on the codimension of support
of $L$,  we reduce to showing that for every irreducible $L\in \MM$ with $\supp(L)\subsetneq X$,
there exists
a simple root $\al$, such that $\supp(R_\al(L))\supsetneq \supp (L)$.
It is easy to see that this is the case when $\dim (\supp(L))=\dim (\pi_\al(\supp(L))$,
which obviously does happen for some simple root $\al$ unless $\supp(L)= X$.

It follows from the  above that there exists an $\LL$ such that $\Hom(\Pi_\LL, R_{\al_1}\circ \cdots \circ R_{\al_n}(L))\neq 0$. Thus, by adjunction, $\Hom(R_{\al_m}\circ \cdots \circ R_{\al_1}(\Pi_\LL), L)\neq 0$. Therefore $P_L$ embeds in the projective $R_{\al_m}\circ \cdots \circ R_{\al_1}(\Pi_\LL)$
 proving (a).

Part (b) follows from a general property of translation functors proved in \cite[Proposition 3.5]{BeGi}:
the functor $ T_{\la\to \mu}\circ  T_{\mu\to \la}$ is isomorphic to the functor
$N\mapsto \widetilde U(\fg)\otimes _{\widetilde U(\fg)^{W_\al}} N= Sym(\fh)\otimes_{Sym(\fh)^{W_\al}}N$ (In the notation of \cite{BeGi} $W_\al$ is called $W_\mu$). Now,
\begin{equation*}
\begin{gathered}
\Tt \circ R_\al (M) = \Tt_{\la \to -\rho} \circ T_{\mu \to \la}\circ T_{\la\to\mu} (M) =  \Tt_{\mu \to -\rho}  \circ T_{\la \to \mu} \circ T_{\mu \to \la}\circ T_{\la\to\mu} (M)
\\
= \Tt_{\mu \to -\rho} ( Sym(\fh)\otimes_{Sym(\fh)^{W_\al}}T_{\la\to\mu} (M)) =  Sym(\fh)\otimes_{Sym(\fh)^{W_\al}}\Tt_{\mu \to -\rho} T_{\la\to\mu} (M)
\\
= pr_\al^* pr_{\al *}\circ \Tt (M)
\end{gathered}
\end{equation*}

Part (c) will be verified in section \ref{singbl_sec}
after the proof of Proposition \ref{singbl_prop}.   

\subsection{Fully faithful translation to $\M_{sing}$: proof of Proposition \ref{translff_prop}}

 We have to show that
\begin{equation}\label{is}
\Hom(P,Q)\ \iso \ \Hom(\widetilde T(P), \widetilde T(Q))
\end{equation}
for $P,Q\in \PP$. 

First we claim that \eqref{is} holds when either $P$ or $Q$ is of the form $T_{{-\rho}\to \lambda}(M)$
for some $M\in \widehat\M_{-\rho}$. This follows directly from results of \cite[\S 3]{BeGi}. 
In more detail, ~\cite[Proposition 3.1]{BeGi} implies that \eqref{is} holds if either $P=\tilde\theta_l^+(M)$
or $Q=\tilde\theta_r^+(M)$ where $\tilde \theta_l^+$ and $\tilde \theta_r^+$ are (the notation of \cite{BeGi} for) the left and the right
adjoint to the extended translation functor $\tilde T$, respectively: for $P=\tilde\theta_l^+(M)$ this is recorded in \cite[Corollary 3.3]{BeGi},
 the case $Q=\tilde \theta_r^+(M)$ is similar. Since $T_{{-\rho}\to \lambda}$ is biadjoint to $T_{\la\to {-\rho}}$ and in view of
 \eqref{Tres} we see that 
 $$ T_{{-\rho}\to \lambda} \cong \tilde \theta_l^+ \circ Ind \cong  \tilde \theta_r^+ \circ coInd,$$
where $Ind$ and $coInd$ are the left and right adjoint to $res$ respectively. Thus the essential image
of $ T_{{-\rho}\to \lambda}$ is contained in those of both $ \tilde \theta_l^+$ and  $\tilde \theta_r^+$
which yields the claim made in the first sentence of the paragraph.

We now turn to the general case. It suffices to prove the statement for a projective generator $Q$. To this end
 it suffices to show that for the projective generator $Q$ we can find an exact sequence of the form 
$$
0\to Q \to T_{{-\rho}\to\lambda} M_1 \to T_{{-\rho}\to\lambda} M_2.
$$
In fact we claim that we obtain such a sequence by setting $M_1=T_{\lambda\to{-\rho}}(Q)$,
$M_2=T_{\lambda\to{-\rho}}(C)$, where $C=\coker(Q\to T_{{-\rho}\to\lambda}(M_1))$.
The maps are the canonical adjunction arrows. One sees easily that we are reduced to showing that the maps
\begin{subequations}
\begin{equation}\label{Q}
Q\to  T_{{-\rho}\to\lambda} T_{\lambda\to{-\rho}}(Q)
\end{equation}
and
\begin{equation}\label{C}
C\to  T_{{-\rho}\to\lambda} T_{\lambda\to{-\rho}}(C)
\end{equation}
\end{subequations}
are injective. 

We prove the statements above for a specific projective generator which we construct as follows. Let us write by $\cQ = U(\fg)\otimes_{U(\fk)}V$, where $V$ is a finite dimensional $K$-representation, $U(\fg)$ acts naturally from the left and $K$ acts diagonally: $k(u\otimes v) = \Ad(k)(u)\otimes kv$. The module $\cQ$ is a projective $(\fg,K)$-module (no conditions on the central character) as it represents the exact functor $\Hom_K(V,\ )$. We get a projective generator of $\cM$ by choosing $V$ so that it includes at least one $K$-type from each irreducible module in $\cM$ and then setting $Q \in \Mh$ to be the formal completion of $\cQ$. 
More precisely, the pro-object $Q\in \Mh$ is defined by $Q=\varprojlim\limits_n (\cQ/\fm_\la^n \cQ)_\MM$, where $\fm_\la$ is the ideal in the center of the enveloping $Z(\g)$ corresponding to $\la$ and
the subscript $_\MM$ denotes the projection of the $(\g,K)$-module to the block $\MM$. The corresponding module over $U(\g)_{\hat{\la}}$ is clearly 
a direct summand in $\cQ\otimes_{Z(\g)} Z(\g)_{\hat{\la}}$ (see \S \ref{statement} for notation).  

Let us write $\fg = \fk \oplus \fp$ for the Cartan decomposition with respect to $\theta$. By a result of Kostant and Rallis, see \cite[Theorem 15]{KR} and also \cite[Theorem 1.6]{BBG}, $\cQ$ is a free module over the algebra $Sym(\fp)^\fk= Sym(\fa)^{W_\bR}$. Furthermore, the center $Z(\fg)$ of $U(\fg)$ acts on $\cQ$ via the projection $Sym(\fh)^{W}\to Sym(\fa)^{W_\bR}$, i.e., via the closed  subscheme $\fa^*/W_\RE\subset \h^*/W$. In particular, $Q = \cQ_{\widehat \lambda}$ is torsion free over $\cO(\fa^*/W_\RE)_{\widehat \lambda}$. Note also that the considerations above show that the center $Z(\fg)$ acts on any finitely generated $(\fg,K)$-module via the quotient $Sym(\fa)^{W_\bR}$, i.e., the category of finitely generated $(\fg,K)$-modules is supported on the closed subscheme $\fa^*/W_\RE\subset \h^*/W$. 

We have the following:
\begin{lem}
If $M\in\cM$ is torsion free over $\cO(\fa^*/W_\RE)_{\widehat \lambda}$ then $M\to  T_{{-\rho}\to\lambda} T_{\lambda\to{-\rho}}(M)$ is injective. 
\end{lem}
\begin{proof}
We write $\cK_{\widehat \lambda}$ for the fraction field of $\cO(\fa^*/W_\RE)_{\widehat \lambda}$ and similarly $\cK_{-\widehat {\rho}}$ for the fraction field of $\cO(\fa^*/W_\RE)_{-\widehat {\rho}}$. Let us write  $\cM^{frac}$ for the category we obtain by base changing $\cM$ to  the generic point  $Spec(\cK_{\widehat \lambda})$ via the functor $N\mapsto N_{\cK_{\widehat \lambda}}=N\otimes_{\cO(\fa^*/W_\RE)_{\widehat \lambda}}\cK_{\widehat \lambda}$ and similarly we write $\cM^{frac}_{sing}$ when we apply this functor to $\cM_{sing}$.  The objects in $\cM^{frac}$ are  $U(\fg)\otimes_{Z(\fg)}\cK_{\widehat \lambda}$-modules. Furthermore, the category $\cM^{frac}$ is semisimple as it consists of Harish-Chandra modules at a generic infinitesimal character. In particular, the standard modules in $\cM$ associated to irreducible local systems on orbits attached to the maximally split Cartan become irreducible in $\cM^{frac}$ and all  the other standard modules go to zero. Using the fact that we are working at generic points, we note that the translation functor $T_{\lambda\to{-\rho}}$ from $\cM^{frac}$ to $\cM^{frac}_{sing}$ does not send any objects to zero. 

 Because $M$ is torsion free we have $M\subset M_{\cK_{\widehat \lambda}}$. The discussion above implies that
\begin{equation*}
\text{$M_{\cK_{\widehat \lambda}}$ is a direct summand of $T_{{-\rho}\to\lambda} T_{\lambda\to{-\rho}}(M_{\cK_{\widehat \lambda}})$}\,.
\end{equation*}
Now, 
\begin{equation*}
T_{{-\rho}\to\lambda} T_{\lambda\to{-\rho}}(M_{\cK_{\widehat \lambda}}) = T_{{-\rho}\to\lambda} T_{\lambda\to{-\rho}}(M)\otimes_{\cO(\fa^*/W_\RE)_{\widehat \lambda}}\cK_{\widehat \lambda}
\end{equation*}
Thus, $M\to  T_{{-\rho}\to\lambda} T_{\lambda\to{-\rho}}(M)$ is injective.
\end{proof}

Now, as $Q$ is torsion free over $\cO(\fa^*/W_\RE)_{\widehat \lambda}$ we conclude by the lemma that~\eqref{Q} is injective. Next, to show  injectivity of~\eqref{C} it suffices to show that $C$ is torsion free over $\cO(\fa^*/W_\RE)_{\widehat \lambda}$. Concretely, we have to show that for $f\in \cO(\fa^*/W_\RE)_{\widehat \lambda}$  the multiplication map $C\xrightarrow{ f\cdot}C$ is injective. By a diagram chase this reduces to showing 
that the map $Q/fQ \to T_{{-\rho}\to\lambda} T_{\lambda\to{-\rho}}(Q/fQ)$ is injective. 

 If $Q/fQ  \to T_{{-\rho}\to\lambda} T_{\lambda\to{-\rho}}(Q/fQ )$ is not injective then
its kernel is a nonzero submodule $M\subset Q/fQ $ such that $T_{\lambda\to{-\rho}}(M)=0$. Let us write $\cN$ for the nilpotent cone in $\fg$. It is a standard fact that   $T_{\lambda\to{-\rho}}(N)=0$  for $N\in\cM$ if and only if the dimension of the support variety of $N$, the Gelfand--Kirillov dimension of $N$, is less than $\dim(\BB)$ (see also
Remark \ref{unique_nonvan} below).
Since $G$ is quasi-split we have also  $\dim(\BB)=\dim(\cN\cap \fp)$. 
Thus, for the pro-object $M = \varprojlim M_n$ the Gelfand--Kirillov dimension of each $M_n$ is less than the dimension of the $K$-nilpotent cone $\cN\cap \fp$.

However, using the canonical filtration of $\cQ$ generated by $ 1\otimes V$ over $U(\fg)$, we see that $gr(\cQ) = V\otimes (Sym(\fp))$ as a $Sym(\fg)$-module. Therefore, 
\begin{equation}
gr(Q) = V\otimes (Sym(\fp)\otimes_{\cO(\fa^*/W_\RE)}{\cO(\fa^*/W_\RE)_{\widehat 0}})\,.
\end{equation}
Let us write $\fm$ for the maximal ideal of $\cO(\fa^*/W_\RE)_{0}$ and $Q_n$ for $\cQ_{\widehat \lambda}/\fm^{n+1}$. Then, using the fact that $Sym(\fp)$ is flat over $\cO(\fa^*/W_\RE)$ we see that
\begin{equation}
gr((Q/fQ)_n) = V\otimes (Sym(\fp)\otimes_{\cO(\fa^*/W_\RE)}{\cO(\fa^*/W_\RE)_{0}}/(\fm^{n+1},f))
\end{equation}
is  a Cohen-Macauley module with maximum dimensional support on the  $K$-nilpotent cone $\cN\cap \fp$.  Thus, $(Q/fQ)_n$ has no submodules of lower dimensional support. This implies that $M_n=0$ and hence $M=0$.

\subsection{Cross action and intertwining functors}\label{cross}
We will use the {\em cross action} introduced in  \cite[Definition 8.3.1]{green} on the set of irreducible objects in $\MM$,
denoted by $w:L\mapsto w\times L$ and its relation to 
{\em intertwining functors} $I_w:D^b(\MM)\to D^b(\MM)$ of \cite{BeBe}.
Recall that, as $\la$ is dominant,  the latter can be characterized by:
\begin{equation}\label{Gamma_intertw}
R\Gamma_{w\cdot \la}(I_w(M))\cong R\Gamma_\la(M),
\end{equation} 
where $\Gamma_\mu(M)$ denotes the direct summand
in $\Gamma(X,M)$ on which the abstract Cartan $\fh$ acts through the generalized 
infinitesimal character $\mu$.\footnote{If the $\fh$ action on $\Gamma(X,M)$ 
is diagonalizable then $\Gamma_\la(M)$ can also be described as $\Gamma(G/B,M_\la)$
where $M_\la$ is the corresponding sheaf of modules over the twisted
differential operators ring on $G/B$ corresponding to $\la$.}


We will also need the alternative geometric description of $I_w$.
Assume for simplicity that $w$ is a simple reflection $s_\alpha$. 
Let us write $\B_\alpha$ for the
variety of
pairs $(x',x'')$ in $\B\times \B$ in relative position $s_\alpha$, and $p,q$
for the
natural projections
\begin{equation}
\label{i1}
\CD
\B \ @<{\ p \ }<< \B_\alpha @>{\ q \ }>> \B
\endCD
\end{equation}
to the two factors. As is explained in \cite{SV} section 10 we extend~\eqref{i1} into a commutative diagram
$$
\CD
X \ @<{\ \widehat p \ }<<  Y_\alpha @>{\ \widehat q \ }>> X
\\
@VVV @VVV @VVV
\\
\B \ @<{\ p \ }<< \B_\alpha @>{\ q \ }>> \B
\endCD
$$
In this diagram $Y_\alpha \to \B_\alpha$ is also an $H$-torsor and the maps satisfy:
\begin{equation}
\label{ih}
\widehat p(h\cdot y) \ = \ h\cdot \widehat p(y) \ \ \ \text{and}\ \ \ \widehat
q(h\cdot
 y)
\ =
\ s_\alpha(h)\cdot \widehat q(y)
\end{equation}
for all $ y\in \widehat Y_\alpha$ and $h\in H$. We also note that the two outer squares in the above diagram are Cartesian. The intertwining functors $I_\alpha:D_K(\B)\to D_K(\B)$ and $I_\alpha:D_K^{H_{mon}}(X)\to D_K^{H_{mon}}(X)$ (which by abuse of notation we denote by the same symbol) are given by the formulas
\begin{equation}
I_\al = q_*p^*[1] \qquad \text{or} \qquad I_\al = \widehat q_*\widehat p^*[1]\,;
\end{equation}
here $:D_K^{H_{mon}}(X)$ stands for the category of $K$-equivariant, $H$-monodromic $D$-modules (or constructible sheaves) on $X$.
These functors are known to be equivalences satisfying the braid relations (see \cite{De}\footnote{The discussion in {\em loc. cit.} focuses on sheaves
on $G/B$ rather than monodromic sheaves on $X$ but the same proof applies to that latter case.}  where this observation
is attributed to Brou\' e and Michel \cite{BrMi}, although in some form it goes back to \cite{BeBe}), thus they extend to an action of the braid
group on the two categories; the functor $I_w$ is then the action of the canonical lift of $w$ to the braid group.

Note that $D^b(\M)$ can be realized as a full subcategory in $D_K^{H_{mon}}(X)$.
This follows once one checks that the realization functor \cite{Beil_der_per} 
$D^b(Perv_K^{H-mon}(X))\to D_K^{H_{mon}}(X)$ induces an isomorphism
$$Ext^i_{Perv_K^{H-mon}(X))}(P,Q)\to  Ext^i _{D_K^{H_{mon}}(X)}(P,Q)$$
where $P$, $Q$ run over  sets of generators of the triangulated category. 

Embedding both categories as full subcategories 
in the corresponding categories of pro-objects, we are reduced to showing the isomorphism
for $P$ in a set of generators of $\Mh$ (for various blocks $\M$). Thus we can let
$P$ run over the set of indecomposable projective objects in $\Mh$ and $Q$ run over
the set of costandard objects. In that case $Ext^i(P,Q)=0$ for $i\ne 0$ in both categories:
this holds in $D^b(Perv_K^{H-mon}(X))$ since $P$ is projective and it holds
in $D_K^{H_{mon}}(X)$ since $P$ is filtered by deformed standard
objects: this follows by an inductive construction of a projective cover
of an irreducible $L\in \MM$ parallel to the one in the proof of \cite[Theorem 3.2.1]{BGS}
(see \cite[p. 498]{BGS});
in the present context the step of induction consists in taking
the universal extension of a previously constructed object by a deformed
principal series object $\Pi_L$ instead of a standard object used in \cite{BGS}
(see \cite[\S 5.5]{BeRi1} for a similar construction of deformed tilting objects).
For $i=0$ the isomorphsim holds since $Perv_K^{H-mon}(X)$ is a full
subcategory in $D_K^{H_{mon}}(X)$.

It was pointed out before Proposition \ref{Pgen} that wall crossing functors $R_\al$
preserve the block, in view of \eqref{RtoI} it follows that so do the intertwining functors $I_\al$ and $I_w$, $w\in W$,
i.e.  $I_w$ preserves the subcategory $D^b(\M)\subset D_K^{H_{mon}}(X)$. 
It is clear from the above that the actions of intertwining operators on $D_K^{H_{mon}}(X)$ and on $D_K(\B)$ are compatible via
the pull back functor.

We are primarily interested in local systems on the orbits attached
to the maximally split Cartan. In this case the relationship between the cross action and the intertwining functors is easier to state.
In particular, the following statement characterizes the cross action on 
such local systems uniquely
in terms of the functors $I_w$. It is not used in the body of the paper, but we include it for completeness. Its proof makes use of Proposition~\ref{cross_via_int} below.

For an irreducible object $L\in \MM$ let us write $\Delta_L$ for 
 the standard cover, i.e., the !-extensions of the corresponding local system
 on a $K\times H$-orbit on $X$.
 \begin{Claim}
a)  Assume that $L$ is supported on the open orbit. Then the object
  $I_w(\Delta_L)$ is a perverse sheaf whose support coincides with that
of $w\times L$, and we have:
 $$I_w(\Delta_L)|_{O_{w\times L}}\cong (w\times L)|_{O_{w\times L}}.$$

b) If $L$ is supported on the open orbit then $w\times L$ is supported
on an orbit attached to the maximally split Cartan.
Every irreducible object in $\MM$ supported on such an orbit has
the form $w\times L$ for some $L$ supported on the open orbit.
\end{Claim}

\proof Since this statement will not be used in the rest of the paper, we only include a sketch of the proof.

Part  (a) follows from Proposition \ref{cross_via_int}(a) once we know   that  $I_w(\Delta_L)$ is a perverse sheaf.
The open $K$-orbit in $\BB_0$ is the quotient of $K$ by $K\cap T$ for a maximal torus $T\subset G$,
thus it is affine. This allows one to present $I_w(\Delta_L)$ both as a $*$ and as a $!$ direct
image under an affine morphism, which implies perversity, see, for example, \cite[\S 5.1]{BM} 

Part  (b) follows from the proof of Proposition \ref{cross_via_int}(b).
Indeed,  under our assumptions
the only situation when  the open orbit $O$ in the support of $L$ differs from 
the open orbit $O'$ in  the support of $s_\alpha\times L$ is the complex root situation 
(cases (i,ii) in the proof); in these cases
$O'$ is the only orbit in $\pi_\alpha^{-1}\pi_\alpha(O)$ different from $O$ and it is 
again attached to the maximally split Cartan, which implies the first sentence in part (b) of the Claim.
To verify the second statement observe that for every orbit $O$ except for the open one
there exists a simple root $\alpha$ such that $\dim \pi_\alpha^{-1}\pi_\alpha(O)> \dim (O)$.
As it is pointed out in the proof of  Proposition \ref{cross_via_int}(b), if $O$ belongs
to the maximally split Cartan then this can only happen in the complex root situation
analyzed in case (ii) of the proof. Thus the support of $s_\alpha\times L$ in this case will
have  larger dimension than $L$. Applying induction we find $w\in W$ such that
$w\times L$ is supported on the open orbit. \qed

We will also need some  standard properties of the intertwining functors $I_w$;
we sketch the argument as we were unable to find exact references.

\begin{Prop}\label{int_prop}
a) Let $\mu$ be such that $( \mu+\rho,\alpha) =0$ for a simple root $\al$.
Then  we have a canonical isomorphism

 $$T_{\la\to \mu}(M)
 \cong T_{\la\to\mu} (I_{s_\al}(M)).$$

b) We have an action of $W$
 on $K^0(\MM)$ given by
 $w:[M]\mapsto [I_w(M)]$. We have 
\begin{equation}\label{Trho}
K^0(T_{-\rho\to \la} T_{\la\to -\rho} )= \sum w\in \Zet[W].
\end{equation}

\end{Prop}

\proof a) follows from \eqref{Gamma_intertw} and \cite[Proposition 2.8]{BeGi}
(note that in \cite{BeGi} the requirement that the weights
are dominant is imposed but the statement remains true with the same
proof for not necessarily dominant weights if by $\Gamma_\la$, $\Gamma_\mu$
one understands the corresponding derived global sections functor).

Part (b) can also be deduced from \cite[Proposition 2.8]{BeGi}. In view of
{\em loc. cit.}, for a $D$-module $M\in \MM$ the Lie algebra
module  $T_{-\rho\to \la} T_{\la\to -\rho}(M)$ is given by
$pr_\la (V_{\la+\rho}\otimes R\Gamma_{-\rho}(M))$,
where $pr_\la$ denotes projection to the generalized infinitesimal
central character $\la$ and $V_\nu$ denotes the finite dimensional irreducible $\g$-module
with highest weight $\nu$;  notation $\Gamma_{\la}(M)$ is explained at the beginning of section \ref{cross}. We have a standard filtration on
$V_{\la+\rho}\otimes \O_{G/B}$ with associated graded $\O(\nu)\otimes V_{\la+\rho}[\nu]$; here $V_{\la+\rho}[\nu]$ stands for the $\nu$-weight space.
It induces a filtration on the sheaf $V\otimes M$, which yields an equality in
the Grothendieck group of $(\g,K)$-modules: 
$$
[V_{\la+\rho}\otimes R\Gamma_{-\rho}(M)]=\sum_\nu [V_{\la+\rho}[\nu] \otimes R\Gamma_{-\rho+\nu}(M)].
$$
Applying $pr_\la$ to the right hand side removes the terms corresponding to non-extremal
weights of $V_{\la+\rho}$, in view of  \eqref{Gamma_intertw} the resulting expression
coincides with the right hand side of the equality in \eqref{Trho}.
   \qed

For an irreducible object $L\in \MM$ let $O_L$ denote the open orbit in the support of $L$.
Recall the deformed principal series modules $\Pi_\LL$ introduced in 
Definition \ref{def_pr}. 

\begin{Prop}\label{cross_via_int}
Assume that $O_L$ is attached to the maximally split Cartan.

a) We have an equality in the Grothendieck group $K^0(\M)$:
$[I_w(\Delta_L)]=[\Delta_{w\times L}]$.

b) Let $\mu$ be such that $( \mu+\rho,\alpha) =0$ for a simple root $\al$.
Then we have 
isomorphisms:
\begin{equation}\label{TDcross}
T_{\la\to \mu}(\Delta_{s_\alpha \times L})\cong T_{\la\to \mu} (\Delta_L),\end{equation}
\begin{equation}\label{TPicross}
T_{\la\to \mu}(\Pi_{s_\alpha \times \LL})\cong T_{\la\to \mu} (\Pi_\LL).\end{equation}

Moreover, the isomorphisms \eqref{TDcross}, \eqref{TPicross} lift to isomorphisms
{\em in the quotient category} $\MM/Ker(T_{\la\to\mu})$: 
\begin{equation}\label{IsDelta}
I_{s_\al}(\Delta_{s_\alpha \times L})\cong \Delta_L,
\end{equation}
\begin{equation}\label{IsPi}
I_{s_\al} (\Pi_{s_\alpha \times \LL})\cong \Pi_\LL.
\end{equation} 
\end{Prop}

\proof 
Let us recall that in the definition of the functors $T_{\la\to -\rho}$ and $T_{\la\to \mu}$ we pass to the simply connected cover $G_{sc}$. Thus, in proving the proposition we can assume without loss of generality that $G$ is simply connected. 
Then the perverse sheaves $L$, $\Delta_L$ have trivial monodromy along the fibers of the projection $X\to \B$, so they
are pull-backs of perverse sheaves on $\B$ which we will, by abuse of notation, denote
 by the same symbols.
  In view of the compatibility
between the intertwining functors acting on $D_K(\B)$ and $D_K^{H_{mon}}(X)$ pointed out above, part (a) and \eqref{TDcross} 
follow from the corresponding statements about sheaves on $\B$. We now proceed to prove these statements 
obtaining~\eqref{TPicross} as a consequence. 

 Let $O$ denote the image of $O_L$ in $\B$.
 We will use the notation and terminology of \cite{ic3} which we recall here briefly. We first note that as we are considering orbits attached to the maximally split Cartan there are no non-compact imaginary roots. There are no compact imaginary roots either as the group $\GR$ is quasisplit. Let us now consider the projection $\pi_\alpha:\B\to G/P_\alpha$ where $P_\al$ is the parabolic corresponding to the simple root $\alpha$; we are interested in the restriction $\pi_\alpha|_O$. In the case the root $\alpha$ is complex $\pi_\alpha|_O$ is either an isomorphism onto its image or it is an $\bA^1$-fibration. If $\alpha$ is real then $\pi_\alpha|_O$ is a $\bC^*$-fibration. We say that $\alpha$ is of type I if $\pi_\alpha^{-1}\pi_\alpha(O)-O$ consists of two $K$-orbits and say that it is of type II if it consists of one $K$-orbit. We say that  $L$ (or $\Delta_L$) satisfies the parity condition, or, in the terminology of \cite{ic3}, $s_\alpha\in \tau(\Delta_L)$, if $\cL$ extends to a local system on $\pi_\alpha^{-1}\pi_\alpha(O)$. Otherwise we say that it does not satisfy the parity condition or that $s_\alpha\notin \tau(\Delta_L)$.

a) It follows from the discussion above that the simple root $\alpha$ is either real or complex.  Thus, one of the formulas (b1), (b2), (c2), (d2) or (e) in  \cite[Definition 6.4]{ic3} applies. These formulas relate the cross action of a simple
reflection $s$ to the operator $T_s$. According to \cite{LV} the operator $-T_s$
 coincides with the effect of $I_s$ on the $K$-group. In fact, \cite{ic3} is concerned with the action of these operators
on the $q$-deformed version of the $K$-group acted upon by the Hecke algebra, to pass
to our present setting one needs to specialize the variable $u$ appearing in \cite[\S 6.4]{ic3}
to $u=1$. Also note that the basis elements $\gamma$ considered in {\em loc. cit.} correspond
to classes of the form $(-1)^d[\Delta_L]$, where $d$ is the dimension of support of $L$, while the action of $W$
on the $K$-group differs from the one introduced in Proposition \ref{int_prop} by the sign twist.
Taking into account that for a real type root and $s_\alpha\not \in \tau$, i.e., when $\alpha$ does not satisfy the parity condition, one has $s\times \gamma=\gamma$, see~\cite[Proposition, 8.3.18 f)]{green}, we get the statement.

We proceed to prove b). Note that \eqref{IsDelta}, \eqref{IsPi} imply \eqref{TDcross}, \eqref{TPicross}
 by Proposition \ref{int_prop}a), so we focus on  \eqref{IsDelta}, \eqref{IsPi}.
We split the argument into the same cases as in part a). 

In the first two cases we assume  that $\alpha$ is a complex root and in the remaining two cases we assume that $\alpha$ is a real root. 

i) Let $\al$ be a complex root such that the map $\pi_\alpha|_O$ is an ${\mathbb A}^1$ fibration. Then 
it is easy to see that $I_{s_\al}(\Delta_L)\cong \Delta_{L'}$, where $L'$ is an irreducible associated to a local system on the orbit 
$O'=\pi_{\al}^{-1}\pi_\al(O)\setminus O$ (which is a locally closed codimension one 
subvariety in the closure of $O$). Part (a) shows then that $L'=s_\al\times L$,  
so \eqref{IsDelta} follows. 
It is also easy to see that $I_{s_\al}(\Pi_\LL)\cong \Pi_{\LL'}$, which yields \eqref{IsPi}.

ii) Let $\al$ be a complex root such that the map $\pi_\alpha|_O$ is one to one. In this case
$I_{s_\al}(\Delta_L)$ is supported on the closure of $O'=\pi_{\al}^{-1}\pi_\al(O)\setminus O$
(we note  that $O$ is a locally closed codimension one 
subvariety in the closure of $O'$). More precisely, $I_{s_\al}(\Delta_L)=j_!j'_*(\cL'[\dim O'])$
for a local system $\cL'$ on $O'$,
where $j':O'\to \pi_{\al}^{-1}\pi_\al(O)$ and $j:\pi_{\al}^{-1}\pi_\al(O)\to \B$ are the embeddings.
Again, part (a) shows that $L'=s_\al\times L$ where $L'$ is the irreducible associated to the local system $\cL'$. 
It is a standard fact that $I_{s_\al}^{-1}$ and $I_{s_\al}$ induce the same functor on the quotient
$D^b(\MM/Ker(T_{\la\to\mu}))$, so the statement follows from (i) by
switching the roles of $L'$ and $L$ we arrive at the situation (i).

iii) Let $\al$ be a real root and thus the map $\pi_\alpha|_O$ is a $\Ce ^*$-fibration.  We assume first that the restriction 
of  $\cL=L|_{O_L}$ to a fiber is nontrivial, i.e., $\alpha$ does not satisfy the parity condition. Then it is easy to see that 
$I_{s_\al}(\Delta_L)\cong \Delta_L$,
$I_{s_\al}(\Pi_\LL)\cong \Pi_\LL$. We can again apply part (a) to see that
$L=s_\al\times L$ which yields \eqref{IsDelta}, \eqref{IsPi}.

iv)  Let $\al$ be a real root and thus the map $\pi_\alpha|_O$ is a $\Ce ^*$-fibration.  Now we assume that the restriction 
of  the local system $\cL=L|_{O_L}$ to a fiber is trivial, i.e., that $\alpha$ satisfies the parity condition. Then the object $I_{s_\al}(\Delta_L)$ can be described as follows:  $I_{s_\al}(\Delta_L)=j'_! j_*(\cL')[\dim O]$,
where $j:O\to \pi_\al^{-1}\pi_\al(O)$ and $j':\pi_\al^{-1}\pi_\al(O)
\to \B$ are the embeddings and $\cL'$ is a local system on $O_L$ which is also trivial along the fiber.  We apply part (a) once more to conclude that $(s_\al\times L)_{O_L}=\cL'$. 
We also obtain a canonical map $I_{s_\al}(\Delta_L)\to \Delta_{s_\al\times L}$ as follows. Let us write $\bar \cL$ for the extension of $\cL'$ to a local system on $\pi_\al^{-1}\pi_\al(O)$ and let us write $\cL''= \bar\cL|_{\partial O}$ where $\partial O = \pi_\al^{-1}\pi_\al(O) -O$. Note that $\partial O$ consists of either one or two $K$-orbits depending on whether $\alpha$ is of type I or type II. In any event, we have a canonical morphism obtained as a composition
$$
 j_*(\cL')[\dim O] \to \cL''[\dim \partial O] \to  j_!(\cL')[\dim O]
$$
whose kernel and cokernel are isomorphic to $\bar \cL[\dim O]$. By applying $j'_!$ to this morphism we obtain a canonical morphism  $I_{s_\al}(\Delta_L)\to \Delta_{s_\al\times L}$ whose cokernel and kernel are $j_!\bar \cL[\dim O]$ which lies in the kernel of $T_{\la\to \mu}$. Thus we get  \eqref{IsDelta} and \eqref{IsPi} is checked
in a similar fashion.
 \qed

\subsection{The singular block: proof of Proposition \ref{singbl_prop}}
\label{singbl_sec}

We start by recalling some standard facts.

\begin{Lem}\label{exttr}
a) The extended translation functor is a Serre factorization.

b) The translation functor $T_{\la \to -\rho}$ sends irreducible objects in $\M$ to irreducible objects in $\M_{sing}$ or annihilates them. Every irreducible 
object in $\M_{sing}$ comes by translation from a unique irreducible object  in $\M$.

\end{Lem}

\proof Part (b) follows from (a) and (a) is \cite[Proposition 3.1.]{BeGi}. \qed

\begin{Lem}\label{onespecial}
 Assume that $G$ is adjoint.
 There exists exactly 
 one irreducible object $\cL_0$ in $\MM$ which does not 
go to zero under $T_{\la\to -\rho}$. 
\end{Lem}

\proof We deduce the Lemma from the main result of \cite[Theorem 1.15]{ic4}.

We claim  that the bijection between irreducible objects  in $\M$ and $\LM$
constructed in {\em loc. cit.} sends a module satisfying the property in the statement of the Lemma into a finite
dimensional module. To see this recall that the corresponding duality 
between the Grothendieck groups is compatible the usual $W$-actions up to twisting with the sign
character \cite[Proposition 17.16]{green_book} or \cite[Proposition 13.10]{ic4} (see Corollary \ref{equiv}(iii) for a categorification of that
compatibility). 
 Now the claim follows since the space of  sign invariants of $W$ in $K^0(\LM)^W$ is generated by the
classes of finite dimensional modules in $\widecheck\M$, while the space $K^0(\M)_{W,sgn}$ of  coinvariants in $K^0(\M)$
is isomorphic to the image of $K^0(\M)$ under the map induced by the functor $T_{\la\to -\rho}$: here the first first 
statement is standard and the second one follows from Proposition \ref{int_prop}(b)
(here we refer to the action of $W$ as in Proposition \ref{int_prop}, as pointed out in the proof of Proposition \ref{cross_via_int} 
it differs by tensoring with the sign character
from the one considered in \cite{ic4} and other sources).
It is clear that a block in the category of $(\Lg,\LK)$-modules
can not contain more than one finite dimensional module as $\check G$ is simply connected and therefore $\LK$ is connected,
so the isomorphism class of a $(\Lg,\LK)$-module is uniquely determined by its isomorphism class as a $\Lg$-module. Thus the Lemma follows.
\epf

\begin{rmk}\label{unique_nonvan} Another, more elementary result of Vogan  \cite[Theorem 6.2]{Vogan_max}
provides a direct classification for irreducible modules of maximal Gelfand-Kirillov dimension.
It is a standard fact that for $G$ quasi-split this condition is equivalent to nonvanishing of the image of the module
under $T_{\la\to -\rho}$: one way to see it is by using
Proposition \ref{int_prop}(b) and the fact that for $w\in W$, $M\in \MM$
the virtual module $M-w(M)$ has Gelfand-Kirillov dimension less than $\dim(\BB)$. 

It would be interesting to deduce the Lemma directly from that classification.
\end{rmk}

{\sf Until section \ref{nonadj} we assume that
$G$ is adjoint.}

\medskip

Lemmas \ref{exttr}, \ref{onespecial} show that there exists exactly one irreducible object in $\MM_{sing}$. 
Let $P_{sing}\in \Mh_{sing}$ be its projective cover. 
To prove Proposition \ref{singbl_prop} in the present case it suffices
to construct an isomorphism $End(P_{sing})\cong \hatt{\O(\fa^*)}^{\Wb}$.

Recall the deformed principal series module $\Pi_\cL$, see Definition ~\ref{def_pr}, which we will be denoting by $\Pi_{\LL,\widehat{{\la}}}$ to emphasize the infinitesimal character. We define the corresponding deformed principal series module $\Pi_{\LL,\widehat{-\rho}}\in \MM_{sing}$ by the formula  $\Pi_{\LL,\widehat{-\rho}}=T_{\la\to -\rho}(\Pi_{\LL,\widehat{\la}})$. We make use of deformed principal series modules in the context when $L$ is an irreducible in the block supported on the open orbit $X_0$ and $\cL$ is the corresponding local system. In this case the Lie algebra of  $A_L$  (notation introduced prior
to Definition ~\ref{def_pr}) is $\fa$. Therefore the commutative ring $ \hatt{\O(\a^*)}$ acts naturally on $\EE$ and hence
it acts on $\Pi_{\LL,\widehat{-\rho}}$. The following key statement implies Proposition \ref{singbl_prop}.

\begin{Lem}\label{86} 
a) The action of  $ \hatt{\O(\a^*)}$ on  $Hom_{\MM_{sing}}(\Pi_{\LL,\widehat{-\rho}}, P_{sing})$
makes it a free rank one module. Here  $ \hatt{\O(\a^*)}$ 
acts on  $Hom_{\MM_{sing}}(\Pi_{\LL,\widehat{-\rho}}, P_{sing})$
via its action on $\Pi_{\LL,\widehat{-\rho}}$.

b) We have $\Pi_{\LL,\widehat{-\rho}}\cong P_{sing}^{|\Wb|}.$

c) The group $\Wb(\cL)=Stab_{W^\theta}(\cL)$ acts on $\Pi_{\LL,\widehat{-\rho}}$, so that the natural 
$ \hatt{\O(\a^*)}$ module structure on $\Pi_{\LL,\widehat{-\rho}}$
 is equivariant with respect to $\Wb(\cL)$.
\end{Lem}

\proof 
Let $Q$ denote the projective cover of the irreducible $\cL_0\in \MM$ (introduced in Lemma \ref{onespecial}). 
Making use of the identities $\Pi_{\LL,\widehat{-\rho}}=T_{\la\to -\rho}(\Pi_{\LL,\widehat{\la}})$, $T_{-\rho\to \la}(P_{sing})=Q$ we see that
$$
Hom_{\MM_{sing}}(\Pi_{\LL,\widehat{-\rho}}, P_{sing})= Hom_{\MM_{sing}}(T_{\la\to -\rho}\Pi_{\LL,\widehat{\lambda}}, P_{sing})=Hom_{\MM}(\Pi_{\LL,\widehat{\lambda}}, Q)\,.
$$
Writing $j:X_0 \to X$ for the embedding we conclude that 
$$
Hom_{\MM_{sing}}(\Pi_{\LL,\widehat{-\rho}}, P_{sing})=Hom_{D(X_0)}(\cL\otimes \cE, j^*Q)\,.
$$
 
It follows from   \cite{chang}, \cite{Vogan_max} that for  an irreducible local system $\cL$  on $X_0$ belonging to the block $\cM$ the dual principal series module $j_*(\cL)$ contains $\cL_0$ in its Jordan-Hoelder series with multiplicity one.
  On the other hand, for such an $\cL$ we have: 
  $$Ext^{>0}(j^*(Q),\cL)=Ext^{>0}(Q,j_*(\cL))=0,$$
  which implies that $j^*(Q)$ is projective. Thus $j^*(Q)\cong \oplus \cL_i\otimes \EE^{d_i}$,
  for some $d_i\in \Zet_{\geq 0}$, where $\cL_i$ runs over the set of local systems on $X_0$
  belonging to the block. Now we have 
  $$d_i=\dim Hom (j^*Q,\cL_i)= \dim Hom (Q,j_*(\cL_i))= [j_*(\cL_i):\cL_0]=1.$$
  and hence $j^*(Q)= \oplus \cL_i\otimes \EE$. Putting things together we conclude that
  $$
Hom_{\MM_{sing}}(\Pi_{\LL,\widehat{-\rho}}, P_{sing})=End(\cE)=\hatt{\O(\a^*)}\,.
$$
This proves (a).

Because of exactness and adjointness of translation functors they send
projective (pro)objects to projective ones. Since $\Pi_{\LL,\hatt{-\rho}}=T_{\la\to -\rho}(\Pi_{\LL,\hatt{\la}})$
and $\Pi_{\LL,\hatt{\la}}$ is projective, it follows that $\Pi_{\LL,\hatt{-\rho}}$
is projective. Since $\MM_{sing}$ has a unique irreducible object with projective cover $P_{sing}$,
we see that   $\Pi_{\LL,\hatt{-\rho}}\cong P_{\! sing}^{\oplus d}$ for some $d$. 

We have 
\begin{equation}\label{d}
\begin{gathered}
d=\dim Hom (T_{\la\to -\rho}(\Pi_{\LL,\widehat{\la}}), 
T_{\la\to -\rho}(\cL_0))=
\\
\dim Hom (\Pi_{\LL,\widehat{\la}}, T_{-\rho\to \la} T_{\la\to -\rho}(\cL_0) )
= [T_{-\rho\to \la} T_{\la\to -\rho}(\cL_0) : j_{!*}(\cL[\dim X]) ]
\\
=[ T_{-\rho\to \la} T_{\la\to -\rho}(j_! (\cL[\dim X])): j_{!*}(\cL[\dim X]) ],
\end{gathered}
\end{equation}
where the last equality follows from the isomorphism $T_{\la\to -\rho}(\cL_0)\cong
T_{\la\to -\rho}(j_! (\cL[\dim X]))$, which is a consequence of the result of \cite{chang},
\cite{Vogan_max} mentioned above.

By Proposition \ref{int_prop}(b) and Proposition \ref{cross_via_int}(a)
we have an equality in the Grothendieck group $K^0(\MM)$:
 $$[T_{-\rho\to \la} T_{\la\to -\rho}(j_! (\cL[\dim X]))
 )]=\sum M_i,$$ 
 where $M_i$ runs over the set of principal series modules coming from orbits
 attached to the maximally split Cartan, each one appearing with multiplicity 
 $|\Wb|$. (Recall the running assumption that $G$ is adjoint, so $\Wb = \Wbp$).
 Since the irreducible object
 $j_{!*}(\cL[\dim X])$ appears once in the Jordan-Hoelder series of  a principal
 series module coming from the open orbit and does not appear in the 
 other principal series modules, formula \eqref{d} shows that $d=|\Wb|$, which yields
statement (b).
 
 It remains to check (c). First note that  equation \eqref{ih} implies that $I_w$ induces  the automorphism $w$ on
   $\fh\subset End(Id_\MM)$. 
Using the isomorphism \eqref{IsPi} in Proposition \ref{cross_via_int} we see that for $w\in W^\theta$ the isomorphism 
 \eqref{TPicross} is compatible with the  $\hatt{\O(\fa^*)}$-action via the
action of $W^\theta$ on $\a^*$. 

We fix a minimal $K$-type $\psi$ of $T_{\la\to -\rho} (\Delta_L)$. The space $Hom_K(\psi, T_{\la\to -\rho} (\Delta_L))$
is one dimensional by \cite[Section 7]{branching}; for original reference, see \cite{Vogan thesis}.  It is also clear that $Hom_K(\psi, T_{\la\to -\rho} (\Pi_{\LL,\hatt\la}))\cong Hom_K(\psi, T_{\la\to -\rho} (\Delta_L))
\otimes _\Ce \hatt{\O(\fa^*)}$. Thus we see that for $w\in \Wb$ there exists a unique invertible element $f_w$ in  $\hatt{\O(\fa^*)}$
such that composing the isomorphism constructed in \eqref{TPicross} with the action of $f_w$ we obtain
an automorphism of $Hom_K(\psi, T_{\la\to -\rho} (\Pi_{\LL,\hatt\la}))$ such that it restricts  to identity on $Hom_K(\psi, T_{\la\to -\rho} (\Delta_L))$ and to the natural action of $w$ on  $\hatt{\O(\fa^*)}$. It is clear that these automorphisms
produce an action of $\Wb(\cL)$ on $Hom_K(\psi, T_{\la\to -\rho} (\Pi_{\LL,\hatt\la}))$.  Since the automorphisms are compatible with the $\g$-action and the minimal $K$-type generates $T_{\la\to -\rho} (\Pi_{\LL,\hatt\la})$
 we obtain an action of $\Wb(\cL)$ on $T_{\la\to -\rho} (\Pi_{\LL,\hatt\la})$.
\qed

{\em Lemma \ref{86} implies Proposition \ref{singbl_prop}.}
It is clear from Lemma \ref{86}(b) that $\hatt{\O(\a^*)}\cong Hom_{\MM_{sing}}(\Pi_{\LL,\widehat{-\rho}}, P_{sing})$
is an $End(P_{sing})-End(\Pi_{\LL,\widehat{-\rho}})$  bimodule satisfying the second commutant property;
by this we mean that each of the two rings acts faithfully and coincides with the  commutant
of the other ring. Setting $E_1=End(P_{sing})$, $E_2=End(\Pi_{\LL,\widehat{-\rho}})^{op}$, $M=Hom_{\MM_{sing}}(\Pi_{\LL,\widehat{-\rho}}, P_{sing})$
the second commutant property says in particular that $E_1=End_{E_2}(M)$. 
The commutative ring $Z=\hatt{\O(\a^*)}^{W^\theta}$ maps injectively to the centers of $E_1$ and $E_2$.
This map arises from the action of the center of the enveloping algebra. Thus, $E_1$ and $E_2$ can be viewed
as subalgebras in $A=End_Z(M)$.

We know that $E_2\supset \hatt{\O(\a^*)}$ and $M$ is a free rank one module over $\hatt{\O(\a^*)}$;
also by Lemma \ref{86}(c) $\Wb(\cL)$ maps to $E_2$ compatibly with the map from $\hatt{\O(\a^*)}$.
Thus $E_1\imbed \hatt{\O(\a^*)}^{\Wb}$. 
It remains to check that this embedding is actually an isomorphism.
In view of the second commutant property, we have 
$$E_1=(E_1\otimes Frac(\hatt{\O(\a^*)}^{W^\theta})) \cap A,$$
where $Frac$ stands for the fraction field and 
the intersection is taken in $A\otimes Frac(\hatt{\O(\a^*)}^{W^\theta})$.
Thus it suffices to check that the embedding $E_1\imbed \hatt{\O(\a^*)}^{\Wb(\cL)}$
 becomes
an isomorphism after base change to $Frac(\hatt{\O(\a^*)}^{W^\theta})$.

The localization $End(P)_{loc}$ is a subextension of the finite field
extension

\noindent $Frac(\hatt{\O(\a^*)}^{\Wb(\cL)})/Frac(\hatt{\O(\a^*)}^{W^\theta})$. If it was a proper subextension we
would have $$\dim_{End(P_{sing})_{loc}}(Hom(\Pi_{\LL,\widehat{-\rho}}, P_{sing})_{loc})
> \dim _{Frac(\hatt{\O(\a^*)}^{\Wb(\cL)})}  (Frac (\hatt{\O(\a^*)}))= |\Wb|,$$
while Lemma \ref{86}(b) shows that $\rank_{End(P_{sing})_{loc}}(Hom(\Pi_{\LL,\widehat{-\rho}}, P_{sing})_{loc})
=|\Wb|$. \qed

\subsubsection{Proof of Proposition \ref{Pgen}(c)}
For $v\in W^\theta/\Wb$ consider the space 
\begin{equation}
\begin{gathered}
V=Hom_{\widetilde\cM_{sing}}(\widetilde U(\fg)\otimes _{U(\fg)}P_{sing}, \widetilde T(\Pi_{\LL_v,\hatt\lambda}) )
\\
=Hom_{\cM_{sing}}(P_{sing}, T_{\la\to -\rho}(\Pi_{\LL_v,\hatt\lambda}) ).
\end{gathered}
\end{equation}
 It  carries an action 
 of $End(\Pi_{\LL_v,\hatt\lambda})\cong \hatt{\O(\fa^*)}$.
 
 We claim that this action makes $V$ a free rank 1 module over $ \hatt{\O(\fa^*)}$. To 
 see this recall that $V'=Hom_{\cM_{sing}}( T_{\la\to -\rho}(\Pi_{\LL_v,\lambda}),P_{sing})$ is free rank one
 over $ \hatt{\O(\fa^*)}$ by  Lemma \ref{86}(a); in view of Proposition \ref{singbl_prop} and basic
 properties of duality for projective modules over  Cohen-Macaulay rings we have:
$$V=Hom_{\hatt{\O(\fa^*)}^{\Wb(\cL_v)}} (V',\hatt{\O(\fa^*)}^{\Wb(\cL_v)})=Hom_{\hatt{\O(\fa^*)}} (V',\hatt{\O(\fa^*)}),$$
 which shows that $V$ is also free of rank one.

It remains to check that the left action of $ \hatt{\O(\fa^*)}^{\Wb}\cong End(P_{sing})$ and the right action of $ \hatt{\O(\fa^*)}$ on $V$
agree up to the action of $\tilde v\in v\Wb$. 
When $\tilde v=1$ this follows from the construction of the isomorphism $End(P_{sing})\cong \hatt{\O(\fa^*)}^{\Wb}$.
The general case then follows by Lemma \ref{86}(c).
  \qed

\subsection{Non-adjoint groups.}\label{nonadj} We now complete the  proof of Proposition \ref{singbl_prop} and Proposition \ref{Pgen}(c)
by reducing the general case to the case when the group $G$ is adjoint. Recall that the category $\M_{sing}$ consists of $(\fg,K_{sc})$-modules that we obtained by first viewing $\cM$ as $(\fg,K_{sc})$-modules and then translating them to infinitesimal character $-\rho$. We can proceed in the same way for the category $\M^{ad}$. Recall, however, as explained in subsection~\ref{matching}, that the category $\M^{ad}$ has infinitesimal character $\la+\mu$. We  can compose the translation of $\cM$ to $-\rho$ as  $T_{\la\to -\rho} = T_{\la+\mu\to -\rho}\circ T_{\la\to \la+\mu}$. Therefore the categories $\M_{sing}$ and $\M_{sing}^{ad} $ are related in the following manner.  We have a restriction functor 
\begin{equation*}
Res^{(K_{ad})_{sc}}_{K_{sc}} : \M_{sing}^{ad} \to \M_{sing}  \ \text{having an exact bi-adjoint} \ Ind^{(K_{ad})_{sc}}_{K_{sc}} 
\end{equation*}

Our goal is to construct a commutative diagram
$$\begin{CD}
\M_{sing}^{ad}   @>{\cong}>>    \Coh_0({\a^*}/W_\MM')\\
@V{Res^{(K_{ad})_{sc}}_{K_{sc}} }VV         @VV{Av_\fS}V \\
\M_{sing}    @>{\cong}>>   \Coh^\fS_0({\a^*}/W_\MM')
\end{CD}$$
where $Av_\fS$ is the induction functor (adjoint to the forgetful functor $\Coh^{\fS}_0({\a^*}/W_\MM')\to \Coh_0({\a^*}/W_\MM')$).
We have already constructed the equivalence in the top row. Each of the vertical arrows has
an exact bi-adjoint and both categories on the bottom row can be described as the category of modules
over the corresponding monad acting on the corresponding category in the top row.
Thus we need only to check that the equivalence in the top row is compatible with the above
monads. This reduces to the following. 

\begin{Lem}
The equivalence $\M_{sing}^{ad}  \cong  \Coh_0({\a^*}/W_\MM')$ is naturally compatible with the action of
the group $\fS$ on the two categories. Here $\fS$ acts on the left hand
side by twisting a module with a character $\chi\in \fS\cong (\widetilde K_{ad}/K)^*$, while  the action on the right hand side comes from the action of $\fS\cong \Wb/\Wbp$ on $\fB$.
\end{Lem} 

\proof 
We return to the situation of \S \ref{quasisplit}. We let $\hatt{\LL_\MM(X_0)}$,
$\hatt{\LL_\MM^{ad}(X_0)}$ denote the 
formal neighborhood of $\LL_\MM(X_0)$ (respectively, $\LL_\MM^{ad}(X_0)$) 
 in the space of $K$ (resp. $\widetilde K_{ad}$) equivariant $T$-monodromic local 
 systems\footnote{In more classical terms this can be described as the set of characters
 of the real torus $T_\RE$.}
  on $X_0$. 

As pointed out in Remark \ref{Bmoncan}, we have an identification
 \begin{equation} \label{comp}
 \hatt{\a^*}/W_\MM'=\hatt{ \LL_\MM^{ad}(X_0)}/W^\theta
 \end{equation}
 coming from the choice of a base point $\LL_0^{ad}$ in $\hatt{\LL_\MM^{ad}(X_0)}$ stabilized
 by $W_\MM'$.
 
 
  Given $w\in W_\MM$ let $\sigma$ be its image 
 under the surjection $W_\MM\to \fS$, see  \eqref{Wbprime}. Then 
 the automorphism $ L\mapsto w_\sigma^{-1}(L\otimes
 \sigma)$ of the set of local systems preserves the base point $\LL_0^{ad}$ and induces the natural
 action of $w$ on  $\hatt{\a^*}$ identified with the formal neighborhood of $\{ \LL_0^{ad}\}$.
  On the other hand, it induces the twisting action of $\sigma$ on the right hand side of
 \eqref{comp}. This shows the desired compatibility.
   \qed 

\bigskip

\subsection{The principal block for split groups}
In this subsection we sketch an alternative approach to the proof of Proposition
\ref{singbl_sec} in the special case when the group $G_\RE$ is split and the block
under consideration is the principal block. This material is not used in the rest of the paper, so 
the details are omitted.

In the case when the group is split and we consider the principal block, there is a concrete description of indecomposable projective modules $P$ at the singular infinitesimal central
character; this also yields a description of the projective cover of a {\it special}
 irreducible module at  regular infinitesimal central character, namely it is isomorphic
 to $T_{-\rho\to \la}(P)$. Here by a special irreducible module we mean an irreducible module
 which survives translation to $-\rho$. Such irreducible modules in the principal block correspond to constant sheaves on special closed orbits; we call a closed $K$-orbit $O$ on $\BB$ {\it special}  if for any simple root  $\al$ the dimension of the image of $O$ in the partial flag variety
 $G/P_\al$ equals $\dim O$. 
Another characterization of special orbits is as follows.  
 Let us write $\pi:T^*\BB \to \BB$ for the projection and $\mu:T^*\BB \to \cN$ for the moment map. The special orbits are precisely the closed $K$-orbits such that the image $\mu(T^*_O \BB)$ of the conormal bundle $T^*_O \BB$ meets the regular locus $\cN_{reg}$ of $\cN$. Conversely, if we write $\fg=\fk\oplus\fp$ for the Cartan decomposition then the union of special closed orbits is the closure of  $\pi(\mu^{-1}(\fp\cap\cN_{reg}))$.

To simplify the discussion, let us assume that $K$ is connected. Given a special orbit $O_s$ we write $i_s:O_s=K/B_K\imbed G/B=\BB$, we let $L_s$ denote the irreducible module associated to the trivial local system on $O_s$. To this situation we can associate a weight $\rho_s=i_s^*(\rho_G)-2\rho_K$ of $K$ and consequently an irreducible $K$-representation $V_s$. We have 
\begin{equation*}
\text{ a) $\rho_s$ is a dominant weight of $K$.}
\end{equation*}
and
\begin{equation*}
\begin{gathered}
\text{b) If $s'\ne s$ then $\rho_{s'}$ is not of the form $\rho_s+i_s^*(\la^+)$}
\\
\text{where $\la^+$ is a sum of positive roots for $G$.}
\end{gathered}
\end{equation*}

The representation $V_s$ has the following properties:

\begin{enumerate}
\item The multiplicity of $V_s$  in the irreducible module
$T_{\la\to -\rho}(L_s)$ equals one,
 its multiplicity in any other 
irreducible $(\g,K)$-module with infinitesimal singular central character $-\rho$
equals zero.
\item The representation $V_s$ is {\em fine} in the sense of \cite{mod}, i.e., its restriction
to the group $C[2]$ of order two elements in the maximal torus $C\subset K$ is a direct sum of
distinct characters and the group $W_K$ permutes these characters transitively. 
\end{enumerate}
Thus, we conclude:
\begin{prop}
The projective cover of the irreducible module $T_{\la\to -\rho }(L_s)$
 is given by $P_s= (U(\fg)\otimes_{U(\fk)}V_s)_{-\widehat\rho}$, where the subscript denotes completion 
 at the infinitesimal central character $(-\rho)$. 
\end{prop}

This provides an alternative way to carry out  one of the steps of the proof of Proposition \ref{singbl_sec}.
More precisely, we can deduce from the above description of $P_s$ that
$End(P_s)$ is a free module over the completion of $\O(\a/W_\RE)=\O(\h/W)$ 
(recall that $G$ is assumed to be split) 
of rank $|W/W_b|$. To this end observe that 
$End(U(\fg)\otimes_{U(\fk)}V_s)=Hom_K(V_s,U(\fg)\otimes_{U(\fk)}V_s)$.
The right hand side admits a filtration whose associated graded is
$Hom_K(V_s, \O(\p)\otimes V_s)$, thus it suffices to see that the latter space
is a free module over $\O(\t/W)$ of rank $|W/W_b|=\dim(V_s)$. 

The  space $Hom_K(V_s, \O(\p)\otimes V_s)$ can also be thought of as $End_{\Coh^K(\p)}(\O(\p)\otimes V_s)$.
Recall that $\p=K\cdot\a$, which shows that  $End_{\Coh^K(\p)}(\O(\p)\otimes V_s)$ embeds into
 $End_{\Coh(\a)}(\O(\a)\otimes V_s)$; furthermore, the image is contained
 in the space of endomorphisms $E$ whose action on the fiber at a point $x\in \a$
 commutes with the action of the centralizer $Z_K(x)$. It is not hard to check that
 the image is in fact equal to that space; moreover, it suffices to check the commutation with the  centralizer
for regular elements $x\in \a$. For such an element
 the stabilizes is identified with $\fS$, and $V_s$ splits as a sum of distinct characters
 of $\fS$ due to $V_s$ being fine. This shows that the generic
 rank of $End_{\Coh^K(\p)}(\O(\p)\otimes V_s)$ 
  as an $\O(\t/W)$-module equals $\dim(V_s)=|W/W_b|$, which implies the desired property
  of $End(P_s)$.

\section{The main result}\label{sect4}
In this section we put the descriptions in sections \ref{sect3} and \ref{sect2} together
to get a comparison between the two categories $\M$ and $\LM$.

\subsection{An equivalence of graded versions}
For a graded algebra $A$ let $A^\bu_{dg}$ denote the corresponding differential
graded algebra with zero differential; for a DG-algebra $A^\bu$ let  $DG-mod(A^\bu)$ denote
the subcategory of perfect complexes in the derived category of differential graded modules
over $A^\bu$.

\begin{Thm}\label{1}
a) There exists a graded algebra $A=\oplus A^d$  together with 

i) an equivalence $\M\cong A-mod^{nil}$, where $A-mod^{nil}$ is the category of 
finite length nilpotent $A$-modules.

ii) an equivalence $\LD\cong DG-mod(A^\bu_{dg})$ sending irreducible perverse sheaves
to direct summands of the free module.

\medskip

b) The algebra $A$ has the following properties. 

The graded components $A^d$ are finite dimensional, $A^d=0$ for $d<0$ and  $A^0$ is a semisimple algebra.

The center of $A$ contains $\O(\fB)^\fS$, the ring of $\fS$ invariant functions on the block variety.
The embedding $\O(\fB)^\fS\subset A$ is compatible with grading, where the grading
on $\O(\fB)$ is such that linear functions on $\fh$, $\fa$ have degree $2$.
\end{Thm}

 Set $\check L=\oplus \check L_i$, where  the $\check L_i$ run over the set of isomorphism classes of irreducible perverse sheaves in the block $\LM$.  We also write $P$ for the direct sum of the proprojective covers of the irreducible objets in $\cM$.  Theorem \ref{1} follows directly from the following more concrete statement: we can set $A=Ext^\bu(\check L,\check L)$
 equipped with the natural homological grading, then equivalence in Theorem \ref{1}(a,i) follows from Theorem \ref{2}(b), while
 equivalence in Theorem \ref{1}(a,ii) follows from Theorem \ref{2}(a). The properties listed in Theorem \ref{1}(b) are clear from the definition.

 \begin{Thm}\label{2}
 a) The differential graded algebra $RHom_{\LD}( \check L, \check L)$ is formal.
 
 b) The algebra $End(P)$ is isomorphic to the completion of the graded
 algebra $Ext^\bu( \check L, \check L)$ with respect to the augmentation ideal topology.
 \end{Thm}
 
 Here part (a) follows by a standard weight argument (see for example \cite[Proposition 6]{BF}) 
 from the fact that the $Ext^i_{\LD}( \check L, \check L)$ is pure of weight $i$ where we either consider corresponding
 sheaves on the flag variety $\LBB$ over $\Fqbar$ or work with Hodge $D$-modules, endowing $L$ 
 with a weight zero Weil (respectively, Hodge) structure. Purity of $Ext^i_{\LD}( \check L, \check L)$ follows
 from Theorem \ref{pribl_thm}, since $ \check L$ is pure by \cite{BBD} and  cohomology of a pure complex
 on a proper variety is pure by \cite{Weil2}. Notice that this also applies to equivariant cohomology
 since the classifying stack $B\LK$ has a model which is ind-proper and ind-smooth.

 Part (b) follows from theorems~\ref{pribl_thm} and~\ref{monodromic main theorem}. 
  
\medskip
 
 We refer to e.g. \cite{BGS} for a definition of a 
 graded version of an abelian category. By a graded version, or lift of a triangulated (respectively, abelian)
 category $\CC$ we understand another triangulated (respectively, abelian) category $\widetilde \CC$
 together with an autoequivalence $S$ of $\widetilde \CC$ (the grading shift functor) 
 denoted by $S:M\mapsto M(1)$,
  the "degrading" triangulated (respectively, exact) 
  functor $d:\widetilde \CC\to \CC$ and an isomorphism,
   $d\circ S\cong d$, such that for $M,\, N\in \widetilde \CC$ we have 
   $Hom(d(M),d(N))\cong \bigoplus\limits_\Zet Hom(\widetilde M,\widetilde N(n))$ and $\CC$ is generated 
   as a triangulated category (respectively, under taking subquoitents) by
   the image of $d$.

\begin{Cor}\label{equiv}
There exist graded versions $\LD^{gr}$ of $\LD$, $\M^{gr}$ of $\M$ and 
an $\O(\fh/W)$-linear abelian category
$\wt{\M}$ with a graded version $\wt{\M}^{gr}$,
such that $\M$ is identified with the subcategory of objects in $\wt{\M}$
where the ideal of $0\in \fh/W$ acts nilpotently, while 
$\Mh$ is identified with the completion of $\wt{\M}$ at the ideal of 0.
Furthermore, we have  an equivalence
$\kappa: D^b(\wt{\M}^{gr})\cong \LD^{gr}$, such that

i) $\kappa (M(1))=\kappa(M)(1)[1]$, where $M\mapsto M(1)$ denotes the shift of grading.

ii) $\kappa$ sends indecomposable projective objects in $\wt{\M}^{gr}$ to shifts of irreducible
perverse sheaves in $\LM$. 

iii) 
 $\kappa$ intertwines the structure of a module category over the respective monoidal categories:
the graded lift of $D_\LB(\LG/\LB)$ acting on  $\LD^{gr}$ and the graded lift of
completed unipotently monodromic objects in $D(U\bs G/U)$ acting on $D^b(\Mh^{gr})$, see \cite{BY}.
The two monoidal categories $D_\LB(\LG/\LB)$ and $D(U\bs G/U)$  are equivalent by \cite{BY}; see also \cite{BeGi}. 
 
 In particular, the graded lift of the functor $\F\mapsto \F*C_w$ where $C_w$
 is the irreducible perverse sheaf corresponding to $w\in W$ is intertwined with
 the functor $P_w$ of convolution with the free-monodromic tilting object $\hat{T}_w$
(the projective functor).
 \end{Cor}

{\em Proof.\ }  Define $\M^{gr}$ to be the category of finite length (equivalently, finite 
dimensional) graded modules over the graded algebra $A$ from Theorem \ref{1}.
Define $\wt{\M}$ to  be the category of finitely generated $A$ modules
and $\wt{\M}^{gr}$ to be the category of finitely generated graded $A$ modules.
Likewise, define $ \LD^{gr}$ to be the perfect derived category of dg-modules 
over $A^\bu_{dg}$ equipped with an additional grading for which the differential
has degree zero.

Theorem \ref{1} shows that $\M^{gr}$ and $\LD^{gr}$ are graded versions
of $\M$, $\LD$ respectively, while $\wt{\M}$ is related to $\M$, $\Mh$ as described
above.

The equivalence $D^b(\wt{\M}^{gr})\cong \LD^{gr}$ is clear from the definition.
It sends an indecomposable projective in $\wt{\M}$ to 
an indecomposable summand in a shift of the free $A^\bu_{dg}$-module.
By Theorem \ref{1}(ii) such a summand corresponds to a 
shift of a graded lift of an irreducible perverse sheaf.

To check (iii) one first observes a similar compatibility for the equivalence
between the additive category of projective objects in $\wt{\M}^{gr}$ 
and that of semisimple complexes in $\LD^{gr}$. These categories are acted upon 
by the graded lifts of free monodromic tilting pro-objects in  monodromic $U$-equivariant sheaves
on $G/U$ and of semisimple complexes in $D_{\LB}(\LG/\LB)$ respectively. These two 
monoidal categories are equivalent as a special case of the main result of \cite{BY}, 
the idea of the proof in this case going back to \cite{BeGi}. The compatibility 
is clear from the construction, since both the monoidal equivalence and equivalence
of the module categories is characterized by its compatibility with translation
to the singular block and cohomology functors respectively. 
Now, the compatibility for the equivalence of  triangulated categories
follows since each of the four triangulated categories  is equivalent to
the homotopy category of the corresponding additive category of generators.
\qed

\medskip

 One can summarize the statement of  Corollary \ref{equiv} by saying that 
 $D^b(\M)$ and $\LD$ are {\em Koszul dual} one to the other.
 
 \begin{rmk}
 Notice that the compatibility isomorphism in Corollary \ref{equiv}(iii)
 is compatible with forgetting the grading functors: 
 this follows by the argument of \cite{BY} where it is shown that the identification of 
 the derived constructible category with the homotopy category of free monodromic
 tilting objects and, respectively, with the homotopy category of semisimple complexes
  is compatible with convolution.
 \end{rmk}
 
 \begin{rmk}
 We expect that the graded version  of categories $\M$, $\Mh$ obtained here as a formal consequence of Koszul duality
 admit a more direct, geometric description via the theory of mixed Hodge $D$-modules. In particular, the 
 grading on the $Hom$ spaces between indecomposable projective pro-objects can be interpreted as the weight
 grading for a natural pure Hodge structure on that space; 
 see section \ref{sect5} for a related discussion.
 \end{rmk}
 
 The following important property of the equivalence $\kappa$ shows the relation
 of our result to equality of Kazhdan-Lusztig polynomials established in \cite{ic4}.

 In what follows we will write $\Pi_L$ instead of $\Pi_\cL$ where $L\in \MM$ is
 the minimal extension an irreducible equivariant local system
 $\cL$.
 
 \begin{Prop}\label{cost_to_st} 
 The equivalence $\kappa$ of Corollary \ref{equiv} sends graded lifts of deformed principal
 series
$\Pi_L$  in $\Mh$ to graded lifts (with a homological shift) of costandard objects $\nabla_{\check{L}}$.
Here $L\mapsto \check{L}$
is a bijection between irreducible objects in $\MM$ and in $\LM$.

 \end{Prop}

\proof  By inspection it is clear that for  an irreducible object $L\in \MM$  we have
$$\kappa(\tilde P_L) \cong \tilde{\check{L}}[m],$$
where $P_L$ is a projective (pro)cover of $L$, $\tilde P_L$ is
its graded lift, $\check{L}$ is an irreducible object in $\LM$
and the integer $m$ depends on the choice of a graded lift.
The map $L\mapsto \check{L}$ is clearly a bijection between
irreducible objects in $\MM$ and in $\LM$.

Furthermore, the support of $L$ is the open $K$-orbit
if and only if support of  $\check{L}$ is closed.
We have
\begin{equation}\label{dim_codim}
\codim(\supp(L))= \dim (\supp(\check{L}))-d,\end{equation}
where $d$ is the dimension of a closed $\LK$-orbit
on $\LBB$. To see this observe that by the proof of Proposition \ref{spec_gen} 
 the right hand side of \eqref{dim_codim} is the minimal length of $w$ such that
for an irreducible $\check{L}_0$ supported on a closed orbit the convolution
 $\check{L}_0*C_w$ contains $\check{L}[m]$ as a direct summand for some $m$ (here $C_w$ is the corresponding irreducible 
in $Perv_\LB(\LBB)$). Likewise, by the proof of Proposition \ref{Pgen},
the left hand side is the minimal length of $w$ such that the projective
cover of $L$ is a direct summand in $P_w(\Pi_{\cL_0})$ where $\cL_0$
is supported on the open orbit and $P_w$ is the corresponding 
projective functor, see Corollary \ref{equiv}(iii).
Thus \eqref{dim_codim} follows from Corollary \ref{equiv}(iii) according to which
 $\kappa$ sends the graded lift of convolution with $C_w$ to that of 
$P_w$.

Let $\LD_{\leq c}$ be the full triangulated subcategory generated
by irreducible objects whose support has dimension 
$d+c$ or less. Let $D^b(\MM)_{\geq c}$ be the full triangulated
subcategory generated by $P_L$ where $L$ runs over irreducible
objects in $\MM$ whose support has codimension $c$ or less.
It follows from \eqref{dim_codim} that $\kappa$ sends
the graded lift of $D^b(\MM)_{\geq c}$ to that of $\LD_{\leq c}$. 

It is not hard to see that $\LD_{\leq c}$ is also generated 
by costandard objects $\nabla_{\check L}$, where $\check L$ runs irreducible objects whose support has dimension 
$d+c$ or less, while $D^b(\MM)_{\geq c}$ is generated by $\Pi_L$,
where $L$ runs over irreducible
objects in $\MM$ whose support has codimension $c$ or less.
Moreover,
\begin{equation}\label{PiandP}\Pi_L\in D^b(\MM)_{>c}^\perp\ \ \ \ \&\ \ \Pi_L \cong P_L \mod D^b(\MM)_{>c},
\end{equation}
\begin{equation}\label{nabandL}
\nabla_{\check{L}}\in \LD_{<c}^\perp\ \ \ \ \&\ \ \nabla_{\check{L}}\cong \check{L}\mod \LD_{<c},
\end{equation}
where $c=\codim(\supp(L))$ and  notation $D^\perp=\{M\ |\ Hom(X,M)=0 \ \forall \, X\in D\}$ is used. It is clear
that \eqref{PiandP}, \eqref{nabandL} characterize $\Pi_L$, $\nabla_{\check{L}}$
uniquely and a similar characterization applies to their graded lifts and yields the statement. \qed

\begin{Rem}\label{KLpol}
By standard BGG reciprocity, the matrix of transformation between the bases
$\{\tilde \Pi_L\}$ and $\tilde P_L$ in the Grothendieck group $K(\Mh^{gr})$
is inverse to the matrix of transformation between the bases $\{\tilde L\}$
and $\{\tilde \nabla_L\}$ of $K(\MM^{gr})$. Thus Corollary \ref{equiv} together
with Proposition \ref{cost_to_st} imply equality between Kazhdan-Lusztig
polynomials for $\MM$ and inverse Kazhdan-Lusztig polynomials for $\LM$
established in a more general setting in \cite{ic4}. 
\end{Rem}

\subsection{The representation theoretic description of the equivariant category}
 The category $\LD$ can be described in module theoretic terms as follows. Consider the 
abelian category $\widetilde{\LM}$, the principal block of $(\Lg, \LK)$ modules with a fixed  {\em generalized}
infinitesimal central character. 
Let $Q=\bigoplus Q_i$ be a minimal 
 projective generator in the category of pro-objects in that category, here $Q_i$
 are pairwise non-isomorphic indecomposable projectives (projective covers of irreducible objects). Denoting $\LA=End(Q)$
 we get an equivalence
 $\widetilde{\LM}\cong \LA-mod^{nil}$. The action of the center of the enveloping algebra makes
$\LA$ a $Sym(\Lh)$-algebra, where $\Lh$ is the abstract Cartan of $\Lg$.
Consider the DG-algebra $\overline\LA^\bu=\LA\Lotimes _{Sym(\Lh)}\Ce$.
Using e.g. \cite[\S 9.3]{BHeck} it is easy to check the following:

\begin{Prop}\label{dga}
We have $\LD\cong DG-mod^{fl}(\overline\LA^\bu)$, where the upper index $^{fl}$ stands for finite length modules.
\end{Prop}
 To clarify the relation with the result of \cite{S} we need to present a 
slight modification of the above constructions.
Let $S\subset \h^*$ be a vector subspace and also consider the complement $S^\perp \subset \fh=\check\h^*$.

Proposition \ref{dga} allows us to reinterpret Theorem \ref{1} as follows: the triangulated categories of DG-modules over the DG algebras $A\Lotimes _{Sym(\h)}\O(S)$ and $\LA\Lotimes _{Sym(\Lh)}\O(S^\perp)$
are in a Koszul duality relation described in Corollary \ref{equiv}.

Furthermore, assume that $S$ is transversal to $\a^*\subset \h^*$, i.e., $Tor_{>0}^
{\O(\h^*)}(\O(S), \O(\a^*))=0$. The main result of \cite{BBG} implies then that $Tor_{>0}^{Sym(\h)}(A, \O(S))=0$, thus the first of the two
 dual triangulated categories discussed in the previous paragraph is identified
 with $D^b(A_S-mod)$, where $A_S= A\otimes_{\O(\h^*)} \O(S)$. 
 If a similar Tor vanishing condition holds on the dual side, with $A,\,S$
replaced by $\LA, \, S^\perp$, then the second category is $D^b(\LA_{S^\perp}-mod)$,
$\LA_S= \LA\otimes_{\O(\h^*)} \O(S^\perp)$.
Observe that $A_S-mod^{fl}$, $\LA_{S^\perp}-mod^{fl}$ are the categories of Harish-Chandra
modules for $(\g,K)$, respectively $(\Lg,\LK)$, subject to a certain condition on the action of the center
of $U(\g)$.

A particularly nice situation occurs when both Tor vanishing conditions hold simultaneously.
In this case Theorem \ref{1} implies

{\em The algebras $A_S$, $\LA_{S^\perp}$ controlling the categories of Harish-Chandra modules
for $(\g,K)$, respectively $(\Lg,\LK)$, are dual Koszul quadratic algebras.}

We mention two examples of that situation. 

\medskip

{\em Koszul duality for category $O$.}
Suppose that $\GR$ is a complex semisimple group viewed as a real group; by
a slight abuse of notation we denote that complex group by $G$. Then $\cM$ is a block of $(\fg\times\fg,G)$-modules and $\widecheck \cM$ is a block of 
$(\check\fg\times \check\fg,\check G)$-modules where $G$ and $\check G$ act diagonally. If $\fh$ is the universal Cartan of $\fg$ then $\fh\times\fh$ is the universal Cartan of $\fg\times\fg$ and $\fa \subset \fh\times\fh$ is the anti-diagonal. Let $S=\fh^*\times \{0\}\subset \fh^*\oplus \fh^*$.
Then both transversality conditions hold. The categories $A_S-mod$, $\LA_{S^\perp}-mod$
in this case are identified with the usual category $O$ (with, say, a fixed integral regular central 
character and no Cartan diagonalizability condition). The equivalence from the previous
paragraph reduces in this case to the one constructed in \cite{S}.

\medskip

{\em Koszul duality for the principal block in a split group.} Consider the case when when the real form of  $\LG$ is split. In this case  $\check\a^*=\Lh^*$ and therefore the second condition holds for any $S$. Let us choose $S=\h^*$ and then  $S^\perp =0$.Thus, we are back at a special case of the situation considered in the previous sections of the paper.
We obtain the following

\begin{Cor}
In the situation of Theorem \ref{1}, assume that $\LG_\RE$ is split. Then there exist
two Koszul dual graded rings $A$, $A^!$ with $A^!$ finite dimensional such that:
$$\LM\cong 
A^!-mod_{fd} \qquad
\LD \cong D^b(A^!-mod_{fd}) \qquad \M\cong A-mod_{nilp}\,.
$$ 
\end{Cor}

Note that the above considerations show that $D_K(G/B)\cong D^b(Perv_K(G/B))$ when $G$ is a split
semi-simple group. It would be interesting to obtain a topological proof of this fact. 

\section{Further directions: Hodge $D$-modules}\label{sect5}

We were led to the idea of the block variety and the interpretation of the functor $\widetilde T:\M \to \Coh_0^\fS(\fB)$  from Hodge theoretic considerations. We briefly explain the idea. 

 Let us write $\fg=\fk\oplus\fp$ for the Cartan decomposition and let us recall that $T^*X/H =\widetilde \fg$, the Grothendieck simultaneous resolution. We write $\widetilde\fp$ for the inverse image of $\fp$ under the moment map $T^*X/H =\widetilde\fg \to \fg$. We write $\Coh^{\bC^*\times K}({\widetilde\fp})_{0}$ for the category of $\bC^*\times K$-equivariant 
 coherent sheaves on ${\widetilde\fp}$ set theoretically supported on $\widetilde{\fp}\cap \Nt$, the union of conormal bundles of $K$-orbits on $\BB$. 
 
Let $\M_{\rm Ho}$ denote the full subcategory of the category of $K$-equivariant Hodge  
$D$-modules $M$ on $X$, such that $\on{Forg}(M)\in \M$; here $\on{Forg}$ denotes
the functor of forgetting the Hodge structure. 

We have a functor
\begin{equation*}
\gr: \M _{\rm Ho}  \To 
  \Coh^{\bC^*\times K}(\widetilde\fp)_{0}
\end{equation*}
taking a Hodge $D$-module to its associated graded with respect to the Hodge filtration.

\begin{conj}\label{conj1}
a) There exists a full subcategory $\Mmix\subset \M_{\rm Ho}$ which is a {\em graded version}
of $\M$.

b) For $ M,N\in \Mmix$ we have
$$Ext^i(\on{Forg}(M),\on{Forg}(N))\iso Ext^i_{\Coh ^{K}({\widetilde\fp})}(\gr(M), \gr(N))$$
for all $i$.
\end{conj}
In the special case when $G_\RE$ is a complex group the Conjecture will be established 
in \cite{BeRi}, cf. also \S 11.4, Conjecture 56 in \cite{BHeck}.

It may be possible to deduce part (a) of the Conjecture from the result of \cite{AK},
while  \cite{SW} may provide other geometric constructions of 
$\Mmix$.

In order to relate this Conjecture to our present methods we make the following
observation. Let $\Sigma\subset \p$ be a transversal slice to a regular nilpotent orbit, thus
$\Sigma$ is contained in the set of regular elements $\g^{reg}$. Let $\widetilde \Sigma$
denote the set of pairs $(x,\chi)$ where $x\in \Sigma$ and $\chi$ is a character of the
abelian algebraic group $Z_K(x)$, the centralizer of $x$ in $K$. 

One can check that $\widetilde \Sigma$ is naturally equipped with the structure
of an ind-algebraic variety, an infinite union of components, each one of which is a ramified
covering of $\Sigma$.  For $\F\in \Coh^K(\fp)$ each fiber of the sheaf $\F|_\Sigma$ 
is graded by the set of characters of $Z_K(x)$; this observation can be upgraded
to a definition of a coherent sheaf $\F_{\widetilde \Sigma}$ on 
$\widetilde \Sigma$ whose direct image to $\Sigma$
is identified with $\F|_\Sigma$. This clearly extends to a functor $\kappa:\Coh^K(\widetilde \fp)\to \Coh(\widetilde \Sigma\times _{\h^*/W}
\h^*)$ (cf. \cite{Abe} where a similar construction is used in the context of affine Soergel bimodules).

Note that for a regular semisimple element $x\in \Sigma$ the abelian group $Z_K(x)$
can be identified with the stabilizer of a point in the open $K$-orbit $\BB_0$. It is not hard to check the following.

\begin{Lem} 
a) Consider the set of pairs $(x,\chi)\subset \widetilde \Sigma$, where $x$ is regular semisimple 
and $\chi$ is such that the corresponding $K$-equivariant, $T$-monodromic local system on $X_0$
belongs to $\MM$. This is an open subset in a component $\widetilde \Sigma_\MM\subset
\widetilde \Sigma$.

b) We have $\widetilde \Sigma_\MM\cong \a^*/W_\M$.

c) For $M\in \M_{\rm Ho}$ the coherent sheaf $\kappa\circ\gr(M)$ is supported on
$\widetilde \Sigma_\MM\times _{\h^*/W} \h^*$.
\end{Lem}

We can now state:

\begin{conj}
Let $\Mmix$ be as in Conjecture \ref{conj1}. For $M$, $N\in \Mmix$ we have
a canonical isomorphism: $$\kappa\circ \gr(M)\cong \widetilde T(\on{Forg}(M)).$$
\end{conj}

\end{document}